\documentclass{amsart} 
\usepackage{amsmath,amssymb,amsthm}
\usepackage{hyperref}
\usepackage{xcolor}
\usepackage{enumerate}
\usepackage{graphicx}

\calclayout

\def\hat{\widehat}
\def\tilde{\widetilde}
\def \bfo {\begin {eqnarray*} }
\def \efo {\end {eqnarray*} }
\def \ba {\begin {eqnarray*} }
\def \ea {\end {eqnarray*} }
\def \beq {\begin {eqnarray}}
\def \eeq {\end {eqnarray}}
\def \supp {\hbox{supp }}

\def \p {\partial}

\newcommand{\Prob}{{\mathbb P}}

\newcommand{\la}{\lambda}

\newcommand{\R}{{\mathbb R}} 
\newcommand{\Z}{{\mathbb Z}}  
\newcommand{\C}{{\mathbb C}} 
\newcommand{\N}{{\mathbb N}}

\theoremstyle{definition}
\newtheorem{definition}{Definition}[section]

\newtheorem*{construction*}{Construction}
\newtheorem*{notation*}{Notation}

\theoremstyle{theorem}

\newtheorem{theorem}[definition]{Theorem}
\newtheorem{lemma}[definition]{Lemma} 
\newtheorem{proposition}[definition]{Proposition}

\newtheorem{condition}[definition]{Condition}

\theoremstyle{definition}

\numberwithin{equation}{section}

\begin{document} 

\title[Inverse problems for discrete heat equations and random walks]{Inverse problems for discrete heat equations and random walks for a class of graphs}

\author{Emilia Bl{\aa}sten, Hiroshi Isozaki, Matti Lassas and Jinpeng Lu}

\AtEndDocument{\bigskip\medskip{\footnotesize%
  \textsc{Emilia Bl{\aa}sten: Computational Engineering, School of Engineering Science, LUT University, Lahti campus, 15210 Lahti, Finland} \par  
  \textit{Email address}: \texttt{emilia.blasten@iki.fi} \par
  
  \addvspace{\medskipamount}

  \textsc{Hiroshi Isozaki: Graduate School of Pure and Applied Sciences, Professor Emeritus, University of Tsukuba, Tsukuba, 305-8571, Japan} \par  
  \textit{Email address}: \texttt{isozakih@math.tsukuba.ac.jp} \par
  
  \addvspace{\medskipamount}

  \textsc{Matti Lassas: Department of Mathematics and Statistics, University of Helsinki, FI-00014 Helsinki, Finland} \par  
  \textit{Email address}: \texttt{matti.lassas@helsinki.fi} \par
  
  \addvspace{\medskipamount}
  \textsc{Jinpeng Lu: Department of Mathematics and Statistics, University of Helsinki, FI-00014 Helsinki, Finland} \par
  \textit{Email address}: \texttt{jinpeng.lu@helsinki.fi} \par
}}




\begin{abstract}
We study the inverse problem of determining a finite weighted graph $(X,E)$ from the source-to-solution map on a vertex subset $B\subset X$ for heat equations on graphs, where the time variable can be either discrete or continuous. We prove that this problem is equivalent to the discrete version of the inverse interior spectral problem, provided that there does not exist a nonzero eigenfunction of the weighted graph Laplacian vanishing identically on $B$. In particular, we consider inverse problems for discrete-time random walks on finite graphs. We show that under a novel geometric condition (called the Two-Points Condition), the graph structure and the transition matrix of the random walk can be uniquely recovered from the distributions of the first passing times on $B$, or from the observation on $B$ of one realization of the random walk.
\end{abstract}

\maketitle

\section{Introduction}

In this paper, we study the inverse problems for heat equations with the discrete graph Laplacian on finite weighted graphs $(X,E)$, where the time variable can be either discrete or continuous. Suppose we are able to measure the source-to-solution map on a given subset $B\subset X$ of vertices. We aim to reconstruct the graph structure, along with the weights, and recover the potential. In particular, we consider inverse problems for discrete-time random walks on finite graphs. Suppose that we are given the distributions of the first passing times on $B$, or are able to observe one realization of a random walk on $B$. We aim to recover the graph structure and the transition matrix of the random walk.

The inverse problems of recovering network parameters have been widely studied in \emph{network tomography} and \emph{optical tomography}, see e.g. \cite{A,BGB,CCLNY,CGHS,RH}. In many situations, direct measurements of network parameters are not possible and one has to rely on inferential methods to provide estimates or predictability.
The network parameters can often be modeled as weights of graphs. One method of recovering the parameters is to observe random walks on the network as a weighted graph and try to recover the transition matrix, as the transition probabilities are often directly related to the weights.
The inverse problems for random walks were studied in \cite{PGM1,PGM2,MOI,RM}, and applications were considered in optical tomography \cite{G92,G01,GM02,GM04,P}, network tomography \cite{GM06,V}, electrical resistor networks \cite{Conklin,Knox-Moradifam,Lawler-Sylvester} and neuroscience \cite{BC}.
These works use different types of random walk measurements at accessible nodes to recover the transition matrix in different settings, assuming the topology of the network is known.

Determining the network topology from limited measurements has been an elusive research problem. The direct problem of how the geometry affects global properties of random walks has been studied with much more success \cite{Barlow,Lawler,Lovasz}. Typical direct problems for random walks study the connections to the spectrum of the graph Laplacian, to the behavior of resistor networks, and to ways of sampling and exploring large networks which cannot be fit into a computer's memory \cite{Cooper,Doyle-Peter,IndexQuality,Sarkar-Moore}. Random walks on social networks have been used to model the spread of diseases and characterize high-risk individuals \cite{Infections}. For the inverse problem of determining the network structure, results were seen mainly for quantum graphs, with continuous time and space domains, and only when the network is a-priori known to be a tree (graph without cycles). Results for other types of graphs have appeared only very recently, see \cite{Avdonin} and the references therein.

In \cite{BILL}, we studied the discrete version of an inverse spectral problem, where we adopted a formulation analogous to the Gel'fand's inverse problem on manifolds with boundary.
The  Gel'fand's inverse  problem   \cite{Ge} for partial differential equations has been a paradigm problem in the study of the mathematical inverse problems and imaging problems arising from applied sciences. 
The combination of the boundary control method, pioneered by Belishev on domains of $\R^n$
and by Belishev and Kurylev on manifolds \cite{BelKur}, and the Tataru's unique continuation theorem \cite{Tataru1}  gave a solution to the inverse problem of determining the isometry type of a Riemannian manifold from given 
boundary spectral data. Generalizations and alternative methods to solve this problem have been studied in e.g. \cite{AKKLT,Bel-heat,BelKa,Caday,HLOS,KrKL,KOP,LassasOksanen}, see  additional references in \cite{Bel-review,KKL,L}.
 The stability of the solutions of these inverse problems have been analyzed in \cite{AKKLT,BKL3,BILL-stability,FIKLN,FILLN,StU}. 
Numerical methods to solve the Gel'fand's inverse  problems
have been studied in \cite{Bel-num,HoopOksanen1,HoopOksanen2}.
 The inverse problems for the heat, wave and  Schr\"odinger equations on manifolds can be reduced 
 to the Gel'fand's inverse  problem, see \cite{Bel-heat,KKL}. In fact, all these problems are 
 equivalent, see  \cite{KKLM}. In this paper, we also consider the equivalence of the analogous problems
 for discrete graphs. Due to the lack of continuous manifold structure, the inverse problems have
 different nature on graphs than on the smooth manifolds and we provide counterexamples for inverse problems on graphs.

An intermediate model between discrete and continuous models is the quantum graphs (e.g. \cite{BK}), namely graphs equipped with differential operators defined on the edges. In this model, a graph is viewed as glued intervals, and the spectral data that are measured are usually the spectra of differential operators on edges subject to the Kirchhoff condition at vertices. For
such graphs, two problems have attracted much attention. In the case where one uses only the spectra of differential operators as data,  Yurko (\cite{Y05,Y09,Y10}) and other researchers (\cite{AK,BrownW,K08}) have developed so called spectral methods to solve the inverse problems.
Due to the existence of isospectral trees, one spectrum is not enough to determine the operator and therefore multiple measurements are necessary. It is known in \cite{Y10} that the potential can be recovered from appropriate spectral measurements of the Sturm-Liouville operator on any finite graph.
An alternative setting is to  consider inverse problems for quantum graphs when one is given the eigenvalues of the differential operator and the values of the eigenfunctions at some of the nodes. 
Avdonin and Belishev and other researchers (\cite{AK,Avdonin2,Avdonin3,Avdonin4,B04,BV}) have shown that it is possible to solve a type of inverse spectral problem for trees. 
With this method, one can recover both the tree structures and differential operators.
It is worth noting that cycles present significant challenges for this method.

In this paper, we consider inverse problems in the purely discrete setting, that is, for the discrete graph Laplacian.
In this model, a graph is a discrete metric space with no differential structure on edges. The graph can be additionally assigned with weights on vertices and edges. The spectrum of the graph Laplacian on discrete graphs is an object of major interest in discrete mathematics (\cite{C,isospectralgraphs,Analysisongraph,Grigoryan,S08}). It is well-known that the spectrum is closely related to geometric properties of graphs, such as the diameter (\cite{CFM94,CGY,ChungYau}) and the Cheeger constant (\cite{Cheeger,C05,F96}). However, due to the existence of isospectral graphs (\cite{Collatz,isospectralgraphs,FK,Tan}), few results are known regarding the determination of the exact structure of a discrete graph from spectral data. The counterexample for the discrete Gel'fand's inverse problem
 in this paper is obtained by showing that certain examples of  isospectral graphs constructed by K. Fujii and A. Katsuda
 \cite{FK} and  by J. Tan
\cite{Tan} admit  eigenfunctions that coincide at certain vertices.

There have been several studies with the goal of determining the structure or weights of a discrete weighted graph from indirect measurements in the field of inverse problems. These studies mainly focused on the \emph{electrical impedance tomography} on resistor networks (\cite{Bor,CB,CEM89,CM00,LST}),
where electrical measurements are performed at a subset of vertices called the boundary. However, there are graph transformations which do not change the electrical data measured at the boundary, such as changing a triangle into a Y-junction, which makes it impossible to determine the exact structure of the inaccessible part of the network in this way. Instead, the focus was to determine the resistor values of given networks, or to find equivalence classes of networks (with unknown topology) that produce a given set of boundary data (\cite{CdV94,CdVGV96,CIM98}).

\medskip
Our approach follows our recent work \cite{BILL} where we introduced a novel geometric condition called the \emph{Two-Points Condition}. We assumed that the graph structure is unknown but in a class of finite graphs satisfying the Two-Points Condition (with respect to accessible nodes). 
In the setting of \cite{BILL}, 
 the given subset of vertices (accessible nodes) was called the ``boundary vertices'', and we considered equations that were only satisfied on vertices outside of the boundary vertices. The solutions on the boundary vertices were determined by given boundary conditions. We proved that if we could measure the eigenvalues and the eigenfunctions on the boundary vertices, we were able to recover the graph structure and the potential, provided that the graph satisfies the Two-Points Condition.

In the setting of the present paper, we adopt an alternative formulation analogous to the inverse problems on manifolds without boundary. 
Namely, we regard all vertices including the given subset $B$ of vertices as ``interior'', and consider equations that are satisfied everywhere including the subset $B$. 
We assume that the graph structure is unknown but in a class of finite graphs satisfying the Two-Points Condition with respect to $B$. 
From observations of a random walk at $B$, we not only can recover the transition matrix of the random walk but also recover the graph structure (Theorems \ref{prop-hitting times walk} and \ref{prop-walk}). 
The results are proved by relating to the inverse problems for heat equations and the inverse spectral problem. 
Let us remark that this formulation of regarding $B$ as an ``interior'' subset, although different from other works, tends to simplify notations and arguments. As far as our method goes, the ``interior'' formulation does not cause essential difference from a ``boundary'' formulation.

The Two-Points Condition is crucial to determining the graph structure. One important consequence of the condition is that there does not exist a nonzero eigenfunction for the graph Laplacian vanishing identically on $B$ (Proposition \ref{prop-uc}), i.e. the unique continuation property. It is well-known that the absence of the unique continuation property in general poses a main obstacle to uniqueness problems on graphs. Our Two-Points Condition is a verifiable condition for the unique continuation for general graphs, which makes the unique determination of graph structure plausible.

However, the Two-Points Condition does put considerable restriction on the graph structure, more precisely on graphs with cycles, knowing that the Two-Points Condition is satisfied for all trees (with $B$ being all vertices of degree $1$), see \cite[Section 1.3]{BILL}. For periodic lattices and their perturbations, there is a convenient way to test for this condition. Indeed, \cite[Proposition 1.8]{BILL} states that it suffices to search for a 1-Lipschitz ``height'' function with certain properties.
For example, the Two-Points Condition is satisfied for finite subgraphs of square, hexagonal, triangular, graphite lattices and certain types of perturbations (with suitable choices of subset $B$), see \cite[Section 1.3]{BILL}. 
Let us also mention an example where the Two-Points Condition is not satisfied: the Kagome lattice.
Another issue is the stability: we do not know if the determination of graph structure is stable with the current reconstruction method.
We plan to explore other types of conditions and reconstruction methods, and study stability problems in future works.

\subsection{Inverse interior spectral problem} \label{subsection-interior-problem} \hfill

\medskip
The inverse interior spectral problem was originally studied for manifolds, see \cite{BKL3,KrKL}.
Consider the eigenvalues $\lambda_j^{_M}$ and eigenfunctions $\varphi_j$ on a closed Riemannian manifold $(M,g)$ satisfying 
\begin{equation*}
-\Delta_g\varphi_j= \lambda_j^{_M}\varphi_j,
\end{equation*}
where $\Delta_g$ is the Laplace-Beltrami operator on $M$.
In the inverse interior spectral problem, one is  given an open set  $B\subset M$ and the collection of data 
$$\big\{B, \; \lambda_j^{_M}, \; \varphi_j|_{B}, \; j=1,2,\cdots \big\},$$
and the goal is to determine the isometry type of $(M,g)$.

\smallskip
We consider the discrete version of the inverse interior spectral problem. Let $(X,E)$ be a finite undirected simple graph with the vertex set $X$ and the edge set $E$. 
Recall that a graph is simple if there is at most one edge between any pair
of vertices and no edge between the same vertex.
For $x,y\in X$, we denote $x\sim y$ if there is an edge in $E$ connecting $x$ to $y$, that is, $\{x,y\}\in E$. Every vertex $x\in  X$ has a measure (weight) $\mu_x>0$ and every edge $\{x,y\}\in E$ has a symmetric weight $g_{xy}=g_{yx}>0$. For a function $u:X \to \mathbb{R}$, the graph Laplacian $\Delta_X$ on $X$ is defined by
\begin{equation}\label{Laplaciandef_int}
  \big(\Delta_X u\big)(x) = \frac{1}{\mu_x}\sum_{y\sim x,\, y\in X} g_{xy}\big( u(y) - u(x)\big), \quad \hbox{$x\in X$}.
\end{equation}
For our consideration, the weights $\mu,g$ can be chosen arbitrarily and usually related to physical situations.
We mention that two frequent choices of weights in graph theory are 
\begin{eqnarray}
\mu\equiv 1,\ g\equiv 1 &&\textrm{(the combinatorial Laplacian)}; \label{comb-Laplacian}\\
\mu_x=\textrm{deg}(x),\,g\equiv 1 &&\textrm{(the normalized Laplacian)} \label{normalized-Laplacian}. 
\end{eqnarray}

One can equip the space of functions on $X$ with an $L^2(X)$-inner product:
\begin{equation}\label{innerproduct_int}
\langle u_1 ,u_2 \rangle_{L^2(X)} = \sum_{x\in X} \mu_x u_1(x)u_2(x),\quad \textrm{for }u_1,u_2:X\to \R.
\end{equation}
Note that for finite graphs, the function space $L^2(X)$ is exactly the space of real-valued functions on $X$. It is straightforward to check that $\Delta_X$ is self-adjoint with respect to the $L^2(X)$-inner product:
\begin{equation}\label{self-adjoint}
\langle \Delta_X u_1 ,u_2 \rangle_{L^2(X)} =\langle u_1 , \Delta_X u_2 \rangle_{L^2(X)}.
\end{equation}

Let $q: X\to \R$ be potential function. We consider the following eigenvalue problem for the discrete Schr\"odinger operator $-\Delta_X+q$,
\begin{equation}\label{eigenvalue_problem_int}
-\Delta_X\phi_j(x)+q(x)\phi_j(x)= \lambda_j\phi_j(x), \quad\hbox{$x\in X$}.
\end{equation}
We emphasize that in our present formulation, the eigenvalue equation above is satisfied on all vertices, with no boundary or boundary condition involved.
Suppose the eigenfunctions $\phi_j$ are orthonormalized with respect to the $L^2(X)$-inner product so that
\begin{equation*}
\{\phi_j\}_{j=1}^{|X|} \textrm{ is a complete orthonormal family of eigenfunctions in } L^2(X),
\end{equation*}
where $|X|$ denotes the cardinality of the vertex set $X$.

\smallskip
Assume that we are able to measure the eigenfunctions on a subset $B\subset X$ of vertices.

\smallskip
\noindent {\bf Inverse interior spectral problem (i):} Suppose we are given the \emph{interior spectral data} on a subset $B\subset X$,
$$
\big\{B,\; \lambda_j,\;  \phi_j|_{B},\; j=1,2,\dots,|X| \big\}.
$$
Can we determine $(X,E)$ and $\mu$, $g$, $q$?

\begin{figure}[h]
  \begin{center}
    \includegraphics[width=0.85\linewidth]{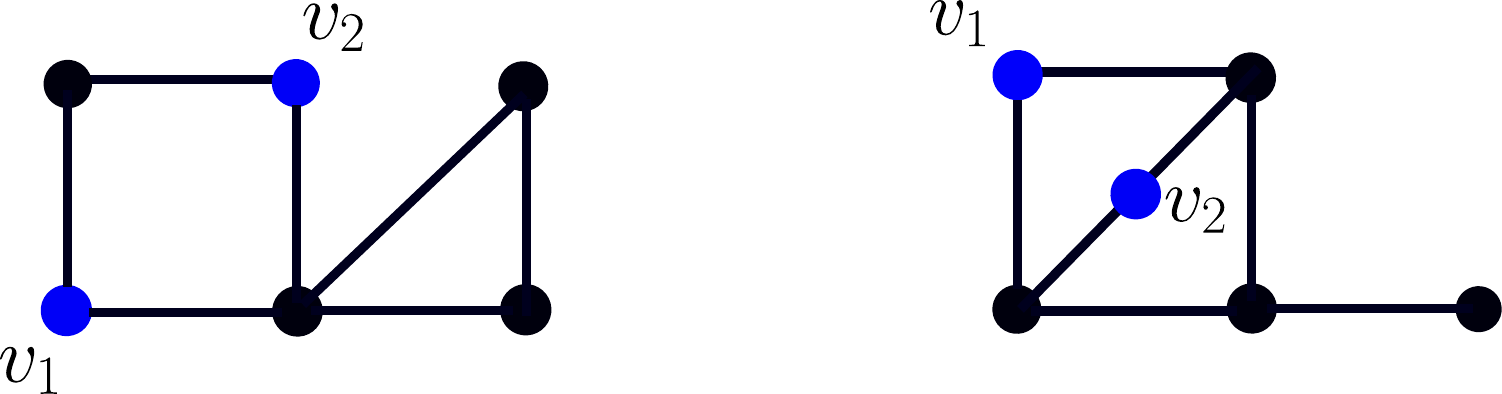}
    \caption{A counterexample for the discrete Gelfand's inverse problem.
    The figure shows two isospectral non-isomorphic graphs for the combinatorial Laplacian \eqref{comb-Laplacian}, on which there exists a complete orthonormal family of eigenfunctions such that their values are identical on the subset $B=\{v_1,v_2\}$ of blue vertices. Furthermore, when the weights are chosen to be $\mu\equiv C\geq 4,\, g\equiv 1$,
 the observations of random walks on the sets $B$ coincide: 
the probabilities $p_{jk}(t)$ that
a random walk (see \eqref{weights-to-walk}) sent at time zero from $v_j\in B$ is observed at time $t$ at $v_k\in B$ are the
same for the two graphs above.
%
     A detailed formulation can be found in Lemma \ref{lemma-isospectral}. This counterexample for the Gel'fand's inverse problem
 is obtained by showing that certain  isospectral graphs constructed in  \cite{FK,Tan}
 have the additional property that their eigenfunctions coincide at suitable vertices.
Note that the subset $B$ is regarded as ``interior'' in our present formulation, in the sense that the equations are also enforced on $B$.}
    \label{fig_isospectral}
  \end{center}
\end{figure}

\smallskip
Solving the inverse interior spectral problem (i) is not possible without further assumptions due to the existence of counterexamples, see Figure \ref{fig_isospectral}. The eigenvalues and corresponding eigenfunctions for this counterexample are shown in Appendix \ref{counterexample_eigenvector}.
We mention that the graphs in Figure \ref{fig_isospectral} are the isospectral graphs for the combinatorial Laplacian \eqref{comb-Laplacian} with the least number of vertices, see \cite{FK,Tan}. 

However, the inverse interior spectral problem can be reduced to the discrete Gel'fand's inverse boundary spectral problem that we have studied in \cite{BILL}. The details of this reduction can be found in Section \ref{subsection-reduction}. In particular, if the Two-Points Condition (Definition \ref{TPC-interior}) is satisfied for $(X,E)$ with respect to $B$, then the inverse interior spectral problem is solvable. Examples of finite graphs satisfying the Two-Points Condition include trees, subgraphs of periodic lattices and their perturbations (see \cite[Section 1.3]{BILL}).

\begin{definition}\label{extreme}
Given a subset $S\subset X$, we say a point $x_0\in S$ is an \emph{extreme point of $S$ with respect to $B$}, if there exists a point $b\in B$ (possibly $b=x_0$) such that $x_0$ is the unique nearest point in $S$ from $b$, with respect to the graph distance on $(X,E)$.
\end{definition}

Here let us recall a few standard definitions. Let $x,y\in X$. A path of $(X,E)$ from $x$ to $y$ is a sequence of vertices $(v_j)_{j=0}^J$ satisfying $v_0=x$, $v_J=y$ and $v_j\sim v_{j+1}$ for $j=0, \ldots, J-1$. The length of the path is $J$. The \emph{graph distance} on $(X,E)$ between $x$ and $y$ is the minimal length among all paths from $x$ to $y$.

\begin{definition}[Two-Points Condition] \label{TPC-interior}
We say a graph $(X,E)$ satisfies the \emph{Two-Points Condition} with respect to a subset $B\subset X$ if the following is true: for any subset $S\subset X$ with cardinality at least $2$, there exist at least two extreme points of $S$ with respect to $B$.
\end{definition}

\begin{theorem} \label{IISP}
Let $(X,E)$ be a finite connected weighted graph with weights $\mu,g$ and $B\subset X$ be a subset. Assume that $(X,E)$ satisfies the Two-Points Condition with respect to $B$. Suppose we are given the interior spectral data for the Schr\"odinger operator $-\Delta_X+q$. Then the graph $(X,E)$ can be reconstructed from the interior spectral data.

Furthermore, the following conclusions hold. \\
(1) If $\mu$ is given by the degree (or a constant), then the weights $\mu,g$ and the potential $q$ can be uniquely recovered from the interior spectral data. \\
(2) If $q=0$, then the weights $\mu,g$ can be uniquely recovered from the interior spectral data.
\end{theorem}

Theorem $\ref{IISP}$ is a direct consequence of Proposition \ref{equi-TPC} and Theorem 1,2 in \cite{BILL}. Another notable consequence of the Two-Points Condition is the unique continuation for eigenfunctions (Proposition \ref{prop-uc}), that is, there does not exist a nonzero eigenfunction, of the Schr\"odinger operator $-\Delta_X+q$ with any potential $q$, vanishing identically on $B$.

\smallskip
\subsection{Discrete-time random walk and heat equation} \hfill

\medskip
We consider the discrete-time random walk on a finite graph $(X,E)$. More precisely, let $H_t^{x_0}$, $t\in \N$ be a discrete-time Markov chain with state space $X$, starting from the vertex $x_0\in X$ at $t=0$ (i.e. $H_0^{x_0}=x_0$). The transition matrix of this Markov chain is defined through the conditional probability of the state changing from $x$ to $y$, i.e.
\begin{equation} \label{pxy}
p_{xy}:=\mathbb{P}(H_{t+1}^{x_0}=y \,|\, H_t^{x_0}=x).
\end{equation}
Assume that $H_t^{x_0}$ is homogeneous so that $p_{xy}$ is independent of $t$.

\begin{condition} \label{condition-Markov}
We impose the following conditions on the random walk $H_t^{x_0}$.
\begin{itemize}
\item[(1)] At any time, the random walk either stays or moves to a connected vertex, i.e. $p_{xy}=0$ unless $x\sim y$ or $x=y$.
\item[(2)] The transition probabilities $p_{xy}$ have the form:
\beq\label{probabilities in c}
p_{xy}=\frac {c_{xy}}{m(x)},\quad m(x)={\sum_{z\sim x\hbox{ \tiny or } z=x} c_{xz}},
\eeq
where $c_{xy}=c_{yx}>0$ for all connected pairs $x\sim y$, and $c_{xx}\geq 0$ for all $x\in X$ (see e.g. \cite[Chapter 1.2]{Barlow}). Note that $c_{xx}$ can be zero.
\end{itemize}
\end{condition}

The form of $p_{xy}$ in \eqref{probabilities in c} may be useful, as the probability $p_{xy}$ is often in a normalized form. For example in an electrical network, the probability can be given by the formula \eqref{probabilities in c}, where $c_{xy}=c_{yx}>0$ denotes the conductance of the edge $\{x,y\}\in E$. For a simpler example, if the random walk moves with equal probability, then one could take $c_{xy}\equiv 1$ and $m(x)=\textrm{deg}(x)$ or $\textrm{deg}(x)+1$, depending on whether the random walk allows staying. 

We remark that it is equivalent to formulate Condition \ref{condition-Markov}(2) in a form $p_{xy}=a(x)b(x,y)$ for some functions $a(x)>0$ and $b(x,y)=b(y,x)>0$. Conversely, if we are given a finite graph $(X,E)$ with \emph{preset} weights $\mu,g$ satisfying $\mu_x \geq \sum_{y\sim x} g_{xy}$, then 
\begin{equation}\label{weights-to-walk}
p_{xy}=\frac{g_{xy}}{\mu_x},\quad p_{xx}=1-\frac{1}{\mu_x}\sum_{y\sim x} g_{xy}
\end{equation}
defines a random walk on $(X,E)$ satisfying Condition \ref{condition-Markov}.
In the inverse problem for random walks, we consider $c_{xy}$ (or functions $a,b$) being unknown. We aim to recover $c_{xy}$ and hence the transition matrix $(p_{xy})_{x,y\in X}$ by observing the random walk on a given subset $B\subset X$.

\smallskip
Given a function $w:X\to \R$, the expectation 
\begin{equation} \label{def-u}
u(x,t):=\mathbb E(w(H_t^{x}))
\end{equation} 
satisfies (see Section \ref{sec-walk} for details)
\begin{eqnarray} \label{eq-pxy-intro}
D_t u(x,t) = \sum_{y\sim x, y\in X}p_{xy} \big(u(y,t)-u(x,t) \big).
\end{eqnarray}
Recall that the discrete time derivative is defined as 
\begin{equation} \label{eq-discrete-derivative}
D_t u(x,t)=u(x,t+1)-u(x,t).
\end{equation}
This shows that $u(x,t)$ given by \eqref{def-u} satisfies the discrete-time heat equation, or a Feynman-Kac type formula:
\begin{eqnarray}
& &D_t u(x,t)-\Delta_X u(x,t)=0,\quad\hbox{for  }(x,t)\in X \times \N, \label{eq-walk-intro}\\
& &u(x,t)|_{t=0}=w(x), \quad\hbox{for  }x\in X, \label{eq-walk-initial-intro}
\end{eqnarray}
with the choice of weights \eqref{weights-walk} in \eqref{Laplaciandef_int}.

\smallskip
Suppose that the graph structure of $(X,E)$ is unknown but we are given a subset $B\subset X$. Next, we consider inverse problems where we observe properties of a random walk on $B$. Consider a random walk $H^{x_0}_t$ starting at $x_0\in B$. For $y\in X$, we define the 
\emph{first passing time}
\beq
\tau^+(x_0,y)=\inf\{t\ge 1: H_t^{x_0}=y\}.
\eeq
Observe that $\tau^+(x_0,y)$ is a random variable taking values in $\mathbb Z_+\cup \{\infty\}$.

\smallskip
\noindent {\bf  Inverse problem for passing times of random walk:} Suppose that we are given 
$B\subset X$ and we are given the distributions of the first passing times $\tau^+(x_0,y)$ for all $x_0,y\in B$,
that is, the probabilities $r(t,x_0,y):=\Prob \big(\{\tau^+(x_0,y)=t\}\big)$  for all $t\in \mathbb Z_+$.
Using these data, can we determine $(X,E)$ and the transition matrix $(p_{xy})_{x,y\in X}$ of the random walk?

\smallskip
Solving the inverse problem above is not possible without further assumptions, as Figure \ref{fig_isospectral} provides a counterexample to this problem, see Lemma \ref{lemma-isospectral}(2). However, we are able to solve this problem under the Two-Points Condition.

\begin{theorem}\label{prop-hitting times walk}
Let $(X,E)$ be a finite connected graph and $B\subset X$ be a subset of vertices.
Assume that $(X,E)$ satisfies the Two-Points Condition with respect to $B$.
We consider a random walk satisfying Condition \ref{condition-Markov}.
Then the probabilities $\Prob\big(\{\tau^+(x_0,y)=t\}\big)$, for all $x_0,y\in B$ and $t\in \mathbb Z_+$,
determine the graph $(X,E)$
and the weights $c_{xy}$ up to a multiple of constant 
$A_0\in \R_+$, that is, the numbers $A_0c_{xy}$
for all $x,y\in X$.
As a consequence, the transition matrix $(p_{xy})_{x,y\in X}$ of the random walk can be uniquely recovered.
\end{theorem}

We also consider inverse problems for a single realization of a random walk.

\smallskip
\noindent {\bf  Inverse problem for a single realization of random walk:} Suppose that we are given $B\subset X$ and we are given the observation on $B$ of one single realization of the random walk $H_t^{x_0},\; t\in \N$ with an unknown starting position $x_0$. Can we determine $(X,E)$ and the transition matrix $(p_{xy})_{x,y\in X}$ of the random walk?

\smallskip
To study this problem, consider a set $\hat X=X\cup\{q_0\}$, where $q_0$  is a new element that does not belong
in $X$. We call $q_0$ ``the point corresponding to an unknown state''
and define, using the random walk $H_t^{x_0}$, a new random process
\beq\label{process hat H}
\hat H_t^{x_0}=\begin{cases}H_t^{x_0},\; \hbox{ if }H_t^{x_0}\in B,\\
q_0,\quad\, \hbox{ if }H_t^{x_0}\in X\setminus B. 
 \end{cases}
\eeq
We say that $\hat H_t^{x_0}$ is the random process given by the observations of $H_t^{x_0}$ in $B$ and denote
$\hat B=B\cup\{q_0\}$.

\begin{theorem}\label{prop-walk}
Let $(X,E)$ be a finite connected graph, $B\subset X$ be a subset of vertices and $q_0\not \in X$.
Assume that $(X,E)$ satisfies the Two-Points Condition with respect to $B$, and that $H_t^{x_0},\; t\in \N$ is
a random walk that satisfies Condition \ref{condition-Markov} and has some  unknown starting position $x_0\in X$.
Suppose we are given $\hat B=B\cup \{q_0\}$  and a single realization of 
 the random process $\hat H_t^{x_0}$.
 Then with probability one, these data uniquely determine
  the graph $(X,E)$ (up to an isometry) and the transition matrix $(p_{xy})_{x,y\in X}$.
\end{theorem}

We note that in Theorem \ref{prop-walk}, the assumption that we know $\hat B$  and
 the random process $\hat H_t^{x_0}$ is equivalent to knowing
$B$, the set $\mathcal T_B=\{t\in \N : H_t^{x_0}\in B\}$ and the sequence
$H_t^{x_0}$ for $t\in\mathcal T_B$.

\bigskip
Now we turn to the general version of the discrete-time heat equation on a finite weighted graph $(X,E)$. Let $\Delta_X$ be the weighted graph Laplacian \eqref{Laplaciandef_int} with weights $\mu,g$, and $q:X\to \R$ be a potential function.
Denote by $U_w$ the solution of the following initial value problem for the discrete-time heat equation
\begin{eqnarray}
& &D_t U_w(x,t)-\Delta_X U_w(x,t)+q(x) U_w(x,t)=0, \quad\hbox{for  }(x,t)\in X \times \N, \label{eq-heat-no-boundary} \\
& &U_w(x,t)|_{t=0}=w(x), \quad\hbox{for  }x\in X. \nonumber
\end{eqnarray}
For the control problem, 
let $U^f: X\times \{0,1,\cdots,T\}\to \R$ be the solution of the following non-homogeneous heat equation up to time $T$,
\begin{eqnarray}
& &D_t U^f(x,t)-\Delta_X U^f(x,t)+q(x)U^f(x,t)=f(x,t), \quad\hbox{for  }(x,t)\in X \times \{0,1,\cdots,T\}, \label{eq-non-homoge-intro} \\
& &U^f(x,t)|_{t=0}=0, \quad\hbox{for  }x\in X , \nonumber
\end{eqnarray}
with a real-valued source $f$. 

\smallskip
We state the following result regarding the observability and controllability at a subset $B\subset X$ for the discrete-time heat equation.

\begin{theorem}\label{Thm-observe-control}
Let $(X,E)$ be a finite weighted graph and $B\subset X$.
Assume that there does not exist a nonzero eigenfunction, of the Schr\"odinger operator $-\Delta_X+q$, vanishing identically on $B$.
Then the following statements hold.
\begin{enumerate}
\item  Suppose we are given the interior spectral data $\big(\lambda_j,\phi_j|_B\big)_{j=1}^{|X|}$. Then the measurement $U_w|_{B\times \N}$ determines $\langle w,\phi_j \rangle_{L^2(X)}$ for all $j=1,\cdots,|X|$.
\item  For any $T\geq |X|,\,T\in \Z_+$, we have
$$\big\{U^f(\cdot,T) : \,\supp(f)\subset B\times \{0,1,\dots,T-1\} \big\}=L^2(X).$$
\end{enumerate}
\end{theorem}

In particular, Theorem \ref{Thm-observe-control} is valid for connected graphs satisfying the Two-Points Condition due to Proposition \ref{prop-uc}.




\smallskip
\subsection{Inverse problems for heat equations} \hfill

\medskip
Let $(X,E)$ be a finite weighted graph with weights $\mu,g$ and $B\subset X$ be a subset of vertices. Let $q:X\to \R$ be a potential function. 
We consider the following inverse problems for the heat equations on graphs. 
For the continuous-time heat equation on $X$, denote by $\Phi^f:X \times \R_{\geq 0} \to \R$ the solution of the following equation
\begin{eqnarray}
& &\frac \p{\p t} \Phi(x,t)-\Delta_X \Phi(x,t)+q(x)\Phi(x,t)=f(x,t),\quad\hbox{for  }(x,t)\in X \times \R_{\geq 0}, \label{eq-cont-heat}\\
& &\Phi(x,t)|_{t=0}=0, \quad\hbox{for  }x\in X, \label{eq-initial-0}
\end{eqnarray}
where $\supp(f) \subset B\times \R_{\geq 0}$.
We define the source-to-solution map on $B$ for the equation above by $\Lambda^c f:=\Phi^f|_{B\times \R_+}$.

We can also consider the continuous-time non-stationary Schr\"odinger equation. Let $\Psi^f:X \times \R_{\geq 0} \to \C$ be the solution of the equation
\begin{eqnarray}
&& i\frac \p{\p t} \Psi(x,t)-\Delta_X \Psi(x,t)+q(x)\Psi(x,t)=f(x,t),\quad\hbox{for  }(x,t)\in X \times \R_{\geq 0}, \label{eq-cont-Schro} \\
& &\Psi(x,t)|_{t=0}=0, \quad\hbox{for  }x\in X, \nonumber
\end{eqnarray}
where the source $f$ is complex-valued and $\supp(f)\subset B\times \R_{\geq 0}$.
Define the source-to-solution map on $B$ for the equation \eqref{eq-cont-Schro} by $\Lambda^s f:=\Psi^f|_{B\times \R_+}$.

\smallskip
For the discrete-time heat equation, denote by $u^f:X \times \N \to \R$ the solution of the equation
\begin{eqnarray}
& &D_t u(x,t)-\Delta_X u(x,t)+q(x)u(x,t)=f(x,t), \quad\hbox{for  }(x,t)\in X \times \N, \label{eq-discrete-heat} \\
& &u(x,t)|_{t=0}=0, \quad\hbox{for  }x\in X, \nonumber
\end{eqnarray}
where $\supp(f) \subset B\times \N$.
Define the source-to-solution map on $B$ for the equation \eqref{eq-discrete-heat} by $\Lambda^d f :=u^f|_{B\times \Z_+}$. We emphasis that the equations above are also enforced on $B$ in our present formulation.

\smallskip
We consider the following inverse problems for the heat equations.

\smallskip
{\bf  Inverse problem (ii):} Given $B$ and $\Lambda^c$, determine $(X,E)$ and $\mu$, $g$, $q$.

\smallskip
{\bf  Inverse problem (iii):} Given $B$ and $\Lambda^s$, determine $(X,E)$ and $\mu$, $g$, $q$.

\smallskip
{\bf  Inverse problem (iv):} Given $B$ and $\Lambda^d$, determine $(X,E)$ and $\mu$, $g$, $q$.

\begin{theorem}\label{Thm-equiv}
Let $(X,E)$ be a finite weighted graph with measure $\mu,g$ and $q: X\to \R$ be a potential function.
Assume that there does not exist a nonzero eigenfunction, of the Schr\"odinger operator $-\Delta_X+q$, vanishing identically on $B$.
If $\mu|_{B}$ is known, then the inverse problems (ii)-(iv) are equivalent to the inverse interior spectral problem (i). 
\end{theorem}

Note that the inverse problems (ii)-(iv) admit a gauge transformation $(\mu,g)\mapsto (c \mu, c g)$, $c\in \R_+$. This is because this gauge transformation does not change the graph Laplacian \eqref{Laplaciandef_int}, and thus does not change the source-to-solution maps. However, the change of $\mu$ affects the normalization of eigenfunctions and therefore changes the interior spectral data.

In particular, the Two-Points Condition (Definition \ref{TPC-interior}) implies the unique continuation for eigenfunctions of the Schr\"odinger operator with any potential and for any choice of weights (Proposition \ref{prop-uc}). Hence the inverse problems (ii)-(iv) are solvable under the Two-Points Condition by Theorem \ref{IISP}, as long as $\mu|_B$ is known.
The analogous result to Theorem \ref{Thm-equiv} for manifolds can be found in \cite{KKL,KKLM}.

\medskip
This paper is organized as follows. In Section \ref{sec-reduction}, we explain the reduction of the inverse interior spectral problem to the inverse boundary spectral problem, and discuss the unique continuation property for eigenfunctions. Sections \ref{sec-heat-boundary} and \ref{sec-walk} are devoted to proving Theorems \ref{Thm-equiv}, \ref{prop-hitting times walk} and \ref{prop-walk}. As an application of the unique continuation for eigenfunctions, we study the observability and controllability for heat equations and prove Theorem \ref{Thm-observe-control} in Section \ref{section-uc}. An example of isospectral graphs with identical interior spectral data is discussed in Appendix \ref{counterexample_eigenvector}.

\medskip
\noindent
{\bf Acknowledgement.}
The authors express their gratitude to anonymous referees for their valuable comments.
H.I. was partially supported by Grant-in-Aid for Scientific Research (C)
20K03667 Japan
Society for the promotion of Science. H.I. is indebted to JSPS.
E.B., M.L.\ and J.L.\ were partially supported by Academy of Finland, grants 273979, 284715, 312110.

\smallskip
\section{Reduction of the inverse interior spectral problem} \label{sec-reduction}

In this section, we show that the inverse interior spectral problem (i) can be reduced to the discrete inverse boundary spectral problem which we have studied in \cite{BILL}. First, let us recall relevant definitions and notations in \cite{BILL}. We note that inverse boundary problems are discussed only in this section, and we will entirely focus on inverse interior problems starting from the next section.

\subsection{A review of the inverse boundary spectral problem} \hfill

\medskip
We say $(G,\partial G,E)$ is a finite graph with boundary if $(G\cup\partial G,E)$ is a finite (undirected simple) graph and $G\cap \partial G=\emptyset$. We call $\partial G$ the set of boundary vertices and call $G$ the set of interior vertices. Note that here the set $\partial G$ (or $G$) is just an arbitrary subset of the whole vertex set.
A pair of connected vertices $x,y\in G\cup\p G$ is denoted by $\{x,y\}\in E$ or simply $x\sim y$. We equip every vertex $x\in G\cup\partial G$ with a measure $\mu_x>0$, and every edge $\{x,y\}\in E$ with a symmetric weight $g_{xy}=g_{yx}>0$. We call a graph with these additional structures a finite weighted graph with boundary, and denote it by $\mathbb{G}=(G,\partial G,E,\mu,g)$.

For a function $u:{{G}}\cup\partial {{G}} \to \mathbb{R}$, the graph Laplacian $\Delta_G$ on ${{G}}$ is defined by
\begin{equation}\label{Laplaciandef}
  \big(\Delta_G u\big)(x) = \frac{1}{\mu_x}\sum_{\substack{y\sim x\\y\in G\cup \partial G}} g_{xy}\big( u(y) - u(x)\big), \quad x\in G,
\end{equation}
and the Neumann boundary value $\partial_{\nu}u$ of $u$ is defined by 
\begin{equation}\label{Neumanndef}
\big(\partial_{\nu} u\big)(z)=\frac{1}{\mu_z}\sum_{\substack{x\sim z\\x\in G}} g_{xz} \big( u(x) - u(z)
  \big), \quad z\in \partial G.
\end{equation}
For $u_1,u_2:G\cup\partial G\to \mathbb{R}$, we define the $L^2(G)$-inner product by
\begin{equation}\label{innerproduct}
\langle u_1 ,u_2 \rangle_{L^2(G)} = \sum_{x\in G} \mu_x u_1(x)u_2(x).
\end{equation}

Let $q:G\to \mathbb{R}$ be a potential function, and we consider the following Neumann eigenvalue problem for the discrete Schr\"odinger operator $-\Delta_G+q$,
\begin{equation} \label{eigenvalueproblem}
    \begin{cases}
      (-\Delta_G+q) \phi_j(x)=\lambda_j \phi_j(x), &x\in G,\\
      \partial_{\nu} \phi_j|_{\partial G}=0 .
    \end{cases}
\end{equation}
Suppose the eigenfunctions $\phi_j$ are orthonormalized with respect to the $L^2(G)$-inner product so that $\{\phi_j\}_{j=1}^{|G|}$ is a complete orthonormal family of eigenfunctions in $L^2(G)$.

\smallskip
\noindent {\bf  Inverse boundary spectral problem:} Suppose we are given $\p G$  and the Neumann boundary spectral data $\big(\lambda_j,\phi_j|_{\p G}\big)_{j=1}^{|G|}$. Can we determine $(G\cup \p G,E)$ and $\mu$, $g$, $q$?

\smallskip
Solving the inverse boundary spectral problem is not possible without further assumptions. A counterexample for the Neumann combinatorial Laplacian can be constructed simply by adding boundary vertices via additional edges connecting to $v_1,v_2$ in Figure \ref{fig_isospectral}. We introduced the Two-Points Condition in \cite{BILL}, and proved that the inverse boundary spectral problem is solvable under the Two-Points Condition (Definition \ref{TPC-b}) and Condition \ref{assumption-boundary-geometry}, see \cite[Theorem 1,2]{BILL}.

\begin{definition}\label{extreme-b}
Let $(G,\partial G,E)$ be a finite graph with boundary.
  Given a subset $S\subset G$, we say a
  point $x_0\in S$ is an \emph{extreme point of $S$ with respect to
    $\partial G$}, if there exists a point $z\in \partial G$ such that
  $x_0$ is the unique nearest point in $S$ from $z$, with respect to
  the graph distance on $(G\cup\partial G,E)$.
\end{definition}

\begin{definition}[Two-Points Condition for graphs with boundary]\label{TPC-b}
We say a finite graph with boundary $(G,\partial G,E)$ satisfies the \emph{Two-Points Condition} if the following is true:
  for any subset $S\subset G$ with
    cardinality at least $2$, there exist at least two extreme points
    of $S$ with respect to $\partial G$.
\end{definition}

\begin{condition} \label{assumption-boundary-geometry}
For any $z\in\partial G$ and any pair of vertices $x,y\in G$, if $x\sim z,\,y\sim z$, then $x\sim y$.
\end{condition}
In particular, Condition \ref{assumption-boundary-geometry} is satisfied if every boundary vertex is connected to only one interior vertex.

\subsection{Unique continuation for eigenfunctions} \hfill

\medskip
One notable consequence of our Two-Points Condition (Definition \ref{TPC-b}) is the unique continuation for the discrete-time wave equation for sufficiently large time.
Next, we consider the unique continuation results for graphs with boundary analogous
to Tataru's unique continuation theorem \cite{Tataru1}, see also \cite{BKL1,BKL2,Laurent-Leautaud}.
On Riemannian manifolds, this 
result states that if a solution $u$ of the wave equation $(\p_t^2-\Delta_g)u=0$ has
vanishing Dirichlet boundary values
$u|_{\Gamma\times (0,T)}=0$ and vanishing Neumann 
 boundary values
$\p_\nu u|_{\Gamma\times (0,T)}=0$, where $\Gamma$ is a non-empty open subset of the boundary, then the solution vanishes inside the manifold in a corresponding causal double cone. Also, if a solution of an elliptic equation, e.g. $-\Delta_g u=\lambda u$, has vanishing Dirichlet and Neumann boundary values on  
a non-empty open subset of the boundary, then the solution $u$ has to be zero.

On a finite weighted graph with boundary $(G,\partial G,E)$, one can consider the following initial value problem of the discrete-time wave equation
\begin{equation}\label{eq-wave}
    \begin{cases}
    D_{tt}  W(x,t)-\Delta_G W(x,t)+q(x)W(x,t)=0, &x\in G,\, t\in \Z_+,\\
      \partial_{\nu} W(x,t)=0, &x\in\partial G,\, t\in \N,\\
      D_t W(x,0) = 0, &x\in G,\\
      W(x,0) = v(x), &x\in G\cup \partial G,
    \end{cases}
\end{equation}
where the discrete time derivative $D_t$ is defined in \eqref{eq-discrete-derivative}, and
$$  D_{tt} u(x,t)=u(x,t+1)-2u(x,t)+u(x,t-1), \qquad t \in \Z_+.$$
We denote by $W^v$ the solution of the wave equation \eqref{eq-wave} with the initial data $v$.

For the wave equation \eqref{eq-wave} on a finite graph, the unique continuation from the boundary may fail, see Figures
\ref{fig_simple_ex}, \ref{fig_larger_ex}. These figures
depict initial values which are also eigenfunctions of the graph Laplacian, and hence the corresponding waves will always be multiples (by the time component) of these initial values. In these particular examples, the boundary values will stay
zero at all times despite the wave being nonzero in the interior.
Our Two-Points Condition guarantees the unique continuation for sufficiently large time, which prevents these examples from appearing.

\begin{figure}
  \begin{center}
    \includegraphics{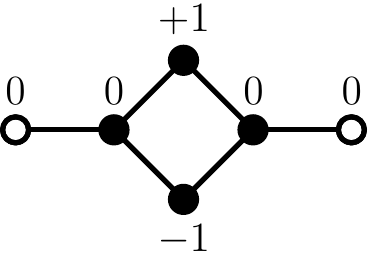}  
     \caption{An example where the unique continuation fails for
     the graph Laplacian. Boundary vertices have white
      centers. We consider the function $\phi: G\cup \p G\to \R$ having the values as indicated in the figure. This function satisfies $-\Delta \phi(x)=\lambda \phi(x)$ for all $x\in G$, where $\Delta$ is the combinatorial Laplacian and $\lambda=2$. The function $\phi$ has vanishing Dirichlet and Neumann boundary values: $\phi|_{\p G}=0$, $\p_\nu \phi|_{\p G}=0$.}
       \label{fig_simple_ex}
  \end{center}
\end{figure}

\begin{figure}
  \begin{center}
    \includegraphics{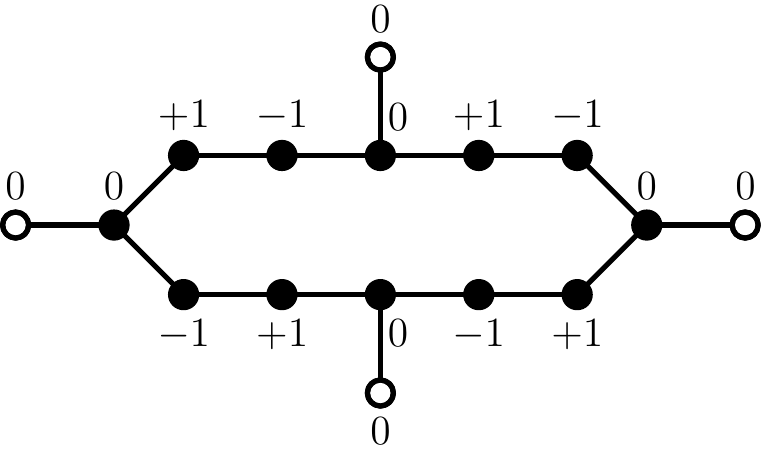}
    \caption{An example where the unique continuation fails for the wave equation with the
      boundary being a resolving set. Boundary vertices have white centers. We consider the function $\phi: G\cup \p G\to \R$ having the values as indicated in the figure, and it satisfies $-\Delta \phi(x)=\lambda \phi(x)$ for all $x\in G$, where $\Delta$ is the combinatorial Laplacian and $\lambda=3$.
With the correct choice of the time component $c(t)$,
   the function $W(x,t)=c(t)\phi(x)$ satisfies the discrete wave equation $D_{tt} W(x,t)-\Delta W(x,t)=0$ and has vanishing Dirichlet and Neumann boundary values: $W|_{\p G\times \mathbb N}=0$,
       $\p_\nu W|_{\p G\times \mathbb N}=0$.}
    \label{fig_larger_ex}
  \end{center}
\end{figure}

\begin{lemma}\label{implications}
  Let $\mathbb G$ be a finite connected weighted graph with boundary, and suppose its reduced graph\footnote{ Recall that the reduced version of a graph is a graph with all its boundary-to-boundary edges removed, see \cite[Definition 2.4]{BILL}.} is connected.
 Assume $\mathbb G$ satisfies the Two-Points Condition (Definition \ref{TPC-b}) and Condition \ref{assumption-boundary-geometry}.
  Then there does not exist a nonzero eigenfunction for the discrete Schr\"odinger operator $-\Delta_G+q$ (with any potential $q$) 
    with both zero Dirichlet and Neumann boundary conditions on
    $\partial G$.
\end{lemma}

\begin{proof}
  We prove that the unique continuation from the
  boundary holds for sufficiently large time, and consequently this prevents the
  existence of eigenfunctions with zero Dirichlet and Neumann
  data. Suppose there exists a nonzero initial value $v:G\cup\partial
  G\to\R$ such that the corresponding wave $W^v$ of
  \eqref{eq-wave} satisfies $W^v|_{\partial G\times \{0,1,\cdots,N\}}=\partial_{\nu}
  W^v|_{\partial G\times \{0,1,\cdots,N\}}=0$, where $N$ is the number of vertices
  in $G$.

  If the initial value $v$ is only supported at one point $x_0\in G$,
  we see that $x_0\in G-N(\partial G)$; otherwise the vanishing boundary
  conditions at time $0$ cannot be satisfied. Then the
  propagation of the wavefront (\cite[Lemma 3.4]{BILL}) implies that
  $W^v(z,d_{re}(x_0,z))\neq 0$ for all $z\in\partial G$, where $d_{re}$ denotes the graph distance function on the reduced graph of $\mathbb{G}$. This contradicts
  the vanishing Cauchy data since $d_{re}(x_0,z)\leq
  N$ for all $z\in\partial G$.

  If the initial value $v$ is supported at multiple points, define
  \[
  S = \{ x\in G \mid v(x)\neq 0 \}.
  \]
  By the Two-Points Condition (Definition \ref{TPC-b}), there exist $x_1\in S$ and $z_1\in
  \partial G$, such that $x_1$ is the unique nearest point in $S$ from
  $z_1$. Note that $v(z_1)=0$ by the vanishing Dirichlet boundary condition.
  Then Lemma 3.4 in \cite{BILL} yields
  $W^v(z_1,d_{re}(x_1,z_1))\neq 0$, which contradicts the vanishing Cauchy data.

  Therefore the initial value $v$ must vanish identically on $G$, and hence $W^v=0$ everywhere at all times. This unique continuation property for
  sufficiently large time implies the lemma. This
  is because if there were such an eigenfunction, one could construct
  a wave with vanishing Cauchy data by multiplying the time
  component to the eigenfunction. This wave would be nonzero but 
  have vanishing Cauchy data on the boundary.
\end{proof}

We remark that the converse of Lemma \ref{implications} is not true. A counterexample is
given by a $3\times 3$ square lattice with $3$ points on one side as
its boundary.

\subsection{Reducing the inverse interior spectral problem to the inverse spectral problem for a graph with boundary} \label{subsection-reduction} \hfill

\medskip
In this subsection, 
let $(X,E)$ be a finite weighted graph with weights $\mu,g$ and $B\subset X$ be a subset. \emph{To distinguish from the notations in previous subsections, we emphasize that the notation $(X,E)$ denotes a standard graph without the notion of the boundary.}
Let $q:X\to \R$ be a potential function and $\big(\lambda_j,\phi_j|_{B} \big)_{j=1}^{|X|}$ be the interior spectral data on $B$.
We reduce the inverse interior spectral problem (i) to the inverse boundary spectral problem as follows.

We take a copy $\tilde{B}$ of the subset $B$ by making a copy $\tilde{b}$ of every vertex $b\in B$. More precisely, we define
\begin{equation}
\tilde{b}:=(b,2), \; b\in B, \quad\; \tilde{B}:=\{\tilde{b}:b\in B\} \subset B\times \Z.
\end{equation}
Consider the following finite graph with boundary:
\begin{equation}\label{copy-boundary}
(X,\tilde{B}, \tilde{E}), \;\textrm{ where }\tilde{E}=E\cup \big\{\{b,\tilde{b}\}: b\in B \big\},
\end{equation}
with the boundary vertex set $\tilde{B}$ and the interior vertex set $X$ (that is, $G=X$, $\partial G=\tilde{B}$, $E=\tilde{E}$ in our notations $(G,\partial G,E)$ for graphs with boundary). 
We keep the measure $\mu$ and the weight $g$ unchanged on $X$, while choose $\mu_{\tilde{b}}$ and $g_{b\tilde{b}}$ arbitrarily for all $\tilde{b}\in \tilde{B}$. 
We consider the Neumann eigenvalue problem \eqref{eigenvalueproblem} on $(X,\tilde{B}, \tilde{E})$:
\begin{eqnarray*}
& &-\Delta_X^{^\mathcal{N}} \tilde{\phi}_j(x)+q(x)\tilde{\phi}_j(x)= \tilde{\lambda}_j \tilde{\phi}_j(x), \quad\hbox{ for $x\in X$}, \\
& & \tilde{\phi}_j(\tilde{b})=\tilde{\phi}_j(b), \quad\hbox{ for $\tilde{b}\in \tilde{B}$},
\end{eqnarray*}
where the graph Laplacian \eqref{Laplaciandef} on $(X,\tilde{B}, \tilde{E})$ is given by\footnote{ The definitions of the graph Laplacians for graphs with and without boundary differ by the interior-to-boundary edges. We add a superscript here to distinguish their notations.}
$$(\Delta_X^{^\mathcal{N}} u)(x):=\frac{1}{\mu_x}\sum_{\{x,y\}\in \tilde{E}} g_{xy}\big( u(y) - u(x)\big), \quad \textrm{for }u: X\cup \tilde{B}\to \R,\; x\in X. $$
Due to the form of the Neumann boundary condition in this specific case, we see that $\tilde{\lambda}_j=\lambda_j$ for all $j$, and
the Neumann eigenfunctions $\tilde{\phi}_j$ on $X\cup \tilde{B}$ are simply extensions of $\phi_j$ on $X$:
\begin{equation}\label{eigenfunction-extension}
\tilde{\phi}_j(x)=\begin{cases} \phi_j(x), \,\textrm{ if }x\in X, \\ \phi_j(b), \,\textrm{ if } x=\tilde{b}\in \tilde{B}. \end{cases}
\end{equation}
This shows that knowing the interior spectral data $\big(\lambda_j,\phi_j|_{B} \big)_{j=1}^{|X|}$ is equivalent to knowing the Neumann boundary spectral data $\big(\tilde{\lambda}_j,\tilde{\phi}_j|_{\tilde{B}} \big)_{j=1}^{|X|}$ on $(X,\tilde{B}, \tilde{E})$. Thus the inverse interior spectral problem on $(X,E)$ is reduced to the inverse boundary spectral problem on $(X,\tilde{B}, \tilde{E})$.

\medskip
The analogous Two-Points Condition for the graph $(X,E)$ (without boundary) with respect to a subset $B\subset X$ is already formulated in Definition \ref{TPC-interior}. In fact, it is equivalent to the one for $(X,\tilde{B}, \tilde{E})$.

\begin{proposition} \label{equi-TPC}
The Two-Points Condition (Definition \ref{TPC-interior}) for $(X,E)$ with respect to $B\subset X$ is equivalent to the Two-Points Condition (Definition \ref{TPC-b}) for the graph with boundary $(X,\tilde{B}, \tilde{E})$.
\end{proposition}
\begin{proof}
This is a direct consequence of the fact that $\tilde{d}(x,\tilde{b})=d(x,b)+1$ for any $x\in X$, where $d$ denotes the graph distance on $(X,E)$ and $\tilde{d}$ denotes the graph distance on $(X\cup\tilde{B}, \tilde{E})$.
\end{proof}

Note that Condition \ref{assumption-boundary-geometry} is automatically satisfied by $(X,\tilde{B}, \tilde{E})$, since every boundary vertex is connected to only one interior vertex by construction. As a consequence, the inverse interior spectral problem is solvable (Theorem \ref{IISP}) under Definition \ref{TPC-interior} by \cite[Theorems 1,2]{BILL}. Analogous to Lemma \ref{implications}, the Two-Points Condition for $(X,E)$ also implies the unique continuation for eigenfunctions.

\begin{proposition} \label{prop-uc}
Let $(X,E)$ be a finite connected weighted graph and $B\subset X$. Assume $(X,E)$ satisfies the Two-Points Condition with respect to $B$ (Definition \ref{TPC-interior}).
Then for any potential $q: X\to \R$, there does not exist a nonzero eigenfunction, of the Schr\"odinger operator $-\Delta_X+q$, vanishing identically on $B$.
\end{proposition}

\begin{proof}
By Proposition \ref{equi-TPC}, the Two-Points Condition for $(X,E)$ with respect to $B$ is equivalent to the Two-Points Condition for $(X,\tilde{B}, \tilde{E})$ defined by \eqref{copy-boundary}. The graph with boundary $(X,\tilde{B}, \tilde{E})$ satisfies Condition \ref{assumption-boundary-geometry}, as every boundary vertex is connected to only one interior vertex by construction. Hence there does not exist a nonzero eigenfunction on $(X, \tilde{B}, \tilde{E})$, for any potential $q$, with zero Dirichlet and Neumann boundary conditions on $\tilde{B}$ by Lemma \ref{implications}. Then the proposition follows from the fact that $\tilde{\phi}_j|_X$ is identical to $\phi_j$ due to \eqref{eigenfunction-extension}.
\end{proof}

\section{Discrete- and continuous-time heat equations} \label{sec-heat-boundary}

In this section, let $(X,E)$ be a finite weighted graph with weights $\mu,g$, and $q: X\to \R$ be a potential function. 
We study the equivalence of inverse problems for the heat equations and prove Theorem \ref{Thm-equiv}. We denote $N=|X|$ in this section.

\begin{lemma}\label{cont-time-heat}
Under the assumption of Theorem \ref{Thm-equiv}, if $\mu|_B$ is known, then the inverse problem (ii) is equivalent to the inverse interior spectral problem (i).
\end{lemma}
\begin{proof}
Let $\{\phi_j\}_{j=1}^N$ be a choice of orthonormalized eigenfunctions of the Schr\"odinger operator $-\Delta_X+q$ corresponding to the eigenvalues $\lambda_j$. By \eqref{self-adjoint},
\begin{eqnarray*}
\langle (-\Delta_X+q)\Phi^f(\cdot,t),\phi_j \rangle_{L^2(X)} &=& \langle \Phi^f(\cdot,t),(-\Delta_X+q)\phi_j \rangle_{L^2(X)}  \\
&=& \lambda_j \langle \Phi^f(\cdot,t),\phi_j \rangle_{L^2(X)} .
\end{eqnarray*}
Then by the heat equation \eqref{eq-cont-heat},
\begin{eqnarray*}
\frac{\p}{\p t} \langle \Phi^f(\cdot,t),\phi_j \rangle_{L^2(X)} &=& \langle (\Delta_X-q) \Phi^f(\cdot,t),\phi_j \rangle_{L^2(X)}+ \langle f(\cdot,t),\phi_j \rangle_{L^2(X)} \\
&=&-\lambda_j \langle \Phi^f(\cdot,t),\phi_j \rangle_{L^2(X)} + \sum_{z\in B} \mu_z\phi_j(z) f(z,t).
\end{eqnarray*}
Note that we have used in the last equality above the support condition of the source: $\supp(f) \subset B \times \R_{\geq 0}$. Hence,
\begin{equation}
\langle \Phi^f(\cdot,t),\phi_j \rangle_{L^2(X)} = \int_{0}^t e^{-\la_j (t-s)} \Big(\sum_{z\in B} \mu_z\phi_j(z) f(z,s) \Big) ds.
\end{equation}
This gives us
\begin{eqnarray*}
\Phi^f(y,t)&=&\sum_{j=1}^N \langle \Phi^f(\cdot,t),\phi_j \rangle_{L^2(X)} \phi_j(y) \\
&=& \sum_{j=1}^N \phi_j(y) \int_{0}^t e^{-\la_j (t-s)} \Big(\sum_{z\in B} \mu_z\phi_j(z) f(z,s) \Big) ds.
\end{eqnarray*}
Thus by definition,
\begin{equation}\label{Lambda-cont-heat}
(\Lambda^c f)(y,t)=\Phi^f(y,t)=\sum_{z\in B} \int_0^t \Big(\sum_{j=1}^N e^{-\la_j (t-s)} \mu_{z} \phi_j(y)\phi_j(z) \Big) f(z,s) ds, \quad y\in B.
\end{equation}
This shows that the interior spectral data determines $\Lambda^c$. 

To see the converse, take $f(z,s)$ to be the form $\delta(s) \delta_{z_1}(z)/\mu_{z_1}$, $z_1\in B$, and thus we can determine the following quantities:
\begin{equation}
Q(z_1,z_2,t):=\sum_{j=1}^N e^{-\la_j t} \phi_j(z_1)\phi_j(z_2),  \quad\forall z_1,z_2\in B,\,t\in \R_+.
\end{equation}
Taking the Laplace transform of $Q(z_1,z_2,t)$ in the time domain $t\in \R_+$, it is possible to determine
\begin{equation}
\widehat{Q}(z_1,z_2,\omega)=\sum_{j=1}^N \frac{1}{\omega+\lambda_j} \phi_j(z_1)\phi_j(z_2),\quad \omega\in \C,\; {\rm Re}(\omega)>\max_j (-\lambda_j).
\end{equation}

Now assume there does not exist a nonzero eigenfunction vanishing identically on $B$. This guarantees that all spectral information can be extracted from the knowledge of $\widehat{Q}$. More precisely, consider
\begin{equation}\label{sum-Qhat}
\sum_{z\in B} \widehat{Q}(z,z,\omega) = \sum_{j=1}^N \frac{1}{\omega+\lambda_j} \Big(\sum_{z\in B}\phi_j(z)^2 \Big), \quad \omega\in \C,\; {\rm Re}(\omega)>\max_j (-\lambda_j).
\end{equation}
Since there does not exist a nonzero eigenfunction vanishing identically on $B$, then $\sum_{z\in B}\phi_j(z)^2$ is nonzero for all $j$. Hence the right-hand side of 
\eqref{sum-Qhat} is a meromorphic function on $\C$ with the set of poles being the set of eigenvalues. We extend the function $\sum_{z\in B} \widehat{Q}(z,z,\omega)$ to a meromorphic function on $\C$ by analytic continuation, and it is possible to determine the poles, i.e. the eigenvalues $-\lambda_j$. Then one can determine the residues of $\widehat{Q}(z_1,z_2,\omega)$ at $w=-\lambda_j$:
\begin{equation}\label{Qj}
Q_j(z_1,z_2):=\sum_{k\in L_j}\phi_k (z_1)\phi_k(z_2), \quad \forall z_1,z_2\in B,
\end{equation}
where 
\begin{equation}
L_j=\{k:\lambda_k=\lambda_j\}.
\end{equation}

For each $j$, the function $Q_j(\cdot,\cdot)$ can be viewed as a $|B| \times |B|$ matrix $Q_j$ defined by $(Q_j)_{kl}=Q_j(z_k,z_l)$. In the matrix form, say $L_j=\{p+1,\cdots,p+|L_j|\}$, we have
$$Q_j=\Big(\phi_{p+1}\big|_{B},\cdots,\phi_{p+|L_j|}\big|_{B}\Big)_{|B|\times |L_j|} \, \Big(\phi_{p+1}\big|_B,\cdots,\phi_{p+|L_j|}\big|_B \Big)_{|B|\times |L_j|}^T\, .$$
Since there does not exist a nonzero eigenfunction vanishing identically on $B$, the vectors $\{\phi_k|_{B}\}_{k\in L_j}$ for each $j$ are linearly independent. Hence the rank of $Q_j$ is simply $|L_j|$.

When $\lambda_j$ has multiplicity $1$, the matrix $Q_j$ determines $\phi_j|_{B}$ up to a multiplication by $-1$. In general, since $Q_j$ is symmetric and positive semi-definite, it can be decomposed into $Q_j=AA^T$, where $A$ is a $|B|\times |L_j|$ matrix of rank $|L_j|$. Moreover, the decomposition is unique up to an $|L_j|\times |L_j|$ orthogonal matrix. Thus we take the column vectors of $A$, and they are the values on $B$ of orthonormalized eigenfunctions found by applying the orthogonal matrix to $\{\phi_k|_B\}_{k\in L_j}$. This gives us a choice of the interior spectral data on $B$. 
\end{proof}

\begin{lemma}\label{cont-time-Schro}
Under the assumption of Theorem \ref{Thm-equiv}, if $\mu|_B$ is known, then the inverse problem (iii) is equivalent to the inverse interior spectral problem (i).
\end{lemma}
\begin{proof}
In the same way as \eqref{Lambda-cont-heat}, \eqref{self-adjoint} and \eqref{eq-cont-Schro} yield
\begin{equation}\label{Lambda-cont-Schro}
(\Lambda^s f)(y,t)=\Psi^f(y,t)=-i\sum_{z\in B} \int_0^t \Big(\sum_{j=1}^N e^{i\la_j (t-s)} \mu_{z} \phi_j(y)\phi_j(z) \Big) f(z,s) ds, \quad y\in B.
\end{equation}
Note that $\Psi^f$ is complex-valued, in which case the $L^2(X)$-inner product \eqref{innerproduct_int} needs to be modified as
$\langle u_1,u_2 \rangle_{L^2(X)} =\sum_{x\in X} \mu_x u_1(x) \overline{u_2(x)}$.

Taking $f(z,s)$ to be of the form $i\delta(s) \delta_{z_1}(z)/\mu_{z_1}$, we can determine
\begin{equation}
F(z_1,z_2,t):=\sum_{j=1}^N e^{i\la_j t} \phi_j(z_1)\phi_j(z_2), \quad\forall z_1,z_2\in B,\,t\in \R_+.
\end{equation}
Taking the Fourier transform of the real part of $F$ in the time domain $t\in \R_+$, it is possible to determine the tempered distribution
\begin{equation}\label{Fhat}
\widehat{F}(z_1,z_2,\xi)=\int_0^{\infty} {\rm Re}\big(F(z_1,z_2,t)\big) e^{-i \xi t} dt=\frac{1}{2}\sum_{j=1}^N \delta(\xi-\lambda_j) \phi_j(z_1)\phi_j(z_2), \quad \xi\in \R.
\end{equation}
Similarly, we consider
\begin{equation}
\sum_{z\in B}\widehat{F}(z,z,\xi)=\frac{1}{2}\sum_{j=1}^N \delta(\xi-\lambda_j) \Big(\sum_{z\in B}\phi_j(z)^2\Big), \quad \xi\in \R.
\end{equation}

Assume there does not exist a nonzero eigenfunction vanishing identically on $B$. Then $\sum_{z\in B}\phi_j(z)^2$ is nonzero for all $j$. 
Choosing test functions of the form $\chi(\xi-a)$, where $\chi\in C_c^{\infty}(\R)$ is supported in a small neighborhood, it is possible to determine the largest  eigenvalue. Then \eqref{Fhat} determines $Q_j(z_1,z_2)$ corresponding to the largest eigenvalue. Repeating this procedure determines all eigenvalues $\lambda_j$ and $Q_j(z_1,z_2)$. The rest of the proof is the same as the last part of Lemma \ref{cont-time-heat}.
\end{proof}

\begin{lemma}\label{discrete-time-heat}
Under the assumption of Theorem \ref{Thm-equiv}, if $\mu|_B$ is known, then the inverse problem (iv) is equivalent to the inverse interior spectral problem (i).
\end{lemma}

\begin{proof}
By \eqref{self-adjoint} and \eqref{eq-discrete-heat},
\begin{eqnarray*}
D_t \langle u^f(\cdot,t),\phi_j \rangle_{L^2(X)} &=& \langle (\Delta_X-q)u^f(\cdot,t),\phi_j \rangle_{L^2(X)}+\langle f(\cdot,t),\phi_j \rangle_{L^2(X)} \\
&=& -\lambda_j \langle u^f(\cdot,t),\phi_j \rangle_{L^2(X)} + \sum_{z\in B} \mu_z\phi_j(z) f(z,t).
\end{eqnarray*}
Thus for $t\in \Z_+$,
$$\langle u^f(\cdot,t),\phi_j \rangle_{L^2(G)} = \sum_{s=0}^{t-1} (1-\lambda_j)^{t-1-s}  \Big(\sum_{z\in B} \mu_z\phi_j(z) f(z,s) \Big).$$
Hence,
\begin{equation}
(\Lambda^d f)(y,t)=u^f(y,t)=\sum_{j=1}^N \sum_{s=0}^{t-1} (1-\lambda_j)^{t-1-s}  \Big(\sum_{z\in B} \mu_{z}\phi_j(z) f(z,s) \Big) \phi_j(y), \quad y\in B,
\end{equation}
which shows that the interior spectral data determines $\Lambda^d$. To see the converse,
take $f(z,s)$ to be the form $\delta_{0}(s) \delta_{z_1}(z)/\mu_{z_1}$, $z_1\in B$, and thus we can determine 
$$Q(z_1,z_2,t)=\sum_{j=1}^N (1-\lambda_j)^{t-1} \phi_j(z_1) \phi_j(z_2),  \quad\forall z_1,z_2\in B,\,t\in \Z_+.$$
Taking the Z-transform (e.g. \cite[Chapter 6.2]{J}) of $Q(z_1,z_2,t)$ in the time domain $t\in \N$, it is possible to determine
$$\mathcal{Z}(Q)(z_1,z_2,\omega)=\sum_{j=1}^N \frac{1}{1-(1-\lambda_j)\omega^{-1}} \phi_j(z_1) \phi_j(z_2), \quad \omega\in \C,\, |\omega|>\max_j |1-\lambda_j|.$$
Similarly as before, we extend the function $\mathcal{Z}(Q)(z_1,z_2,\omega)$ to a meromorphic function on $\C$ by analytic continuation, and then it is possible to determine all eigenvalues $\lambda_j\neq 1$ and the boundary values of corresponding eigenfunctions up to orthogonal transformations.

Now we have already determined all eigenvalues not equal to $1$ and the boundary values of corresponding eigenfunctions. By comparing $Q(y,z,1)$ with $\sum_{\lambda_j\neq 1} \phi_j(y)\phi_j(z)$ for $y,z\in B$, we know whether there exists eigenvalue $1$,
and also know $\sum_{\lambda_k=1}\phi_k(y)\phi_k(z)$. The latter determines the boundary values of the eigenfunctions corresponding to eigenvalue 1 up to an orthogonal transformation.
\end{proof}

Theorem \ref{Thm-equiv} directly follows from Lemmas \ref{cont-time-heat}, \ref{cont-time-Schro} and \ref{discrete-time-heat}.

\smallskip
\noindent {\bf An alternative formulation.} One can also consider the following alternative formulation for the discrete-time heat equation. For a fixed $y_0\in B$, consider
\begin{eqnarray}\label{eq-walk-data}
\begin{cases}
D_t u(x,t)-\Delta_X u(x,t)+q(x)u(x,t)=0, \quad\hbox{for  }(x,t)\in X \times \N,\\
u(x,t)|_{t=0}=\delta_{y_0}(x),  \quad\hbox{for }x\in X.\end{cases}
\end{eqnarray}
We denote the values on $B$ of the solution of the equation above by 
\begin{equation}\label{def-K}
K(y_0;z,t):=u(z,t),\quad z\in B.
\end{equation}

\begin{lemma}\label{lemma-walk-data}
Let $(X,E)$ be a finite weighted graph and $q: X\to \R$ be a potential function.
Then the collection of data $\{K(y;z,t):y,z\in B,t\in \N\}$ determines $\Lambda^d$.
\end{lemma}
\begin{proof}
Recall that $u^f$ denotes the solution of the equation \eqref{eq-discrete-heat} with zero initial value and the non-homogeneous source $f:X\times \N \to \R$, $\supp(f) \subset B\times \N$.
\begin{equation} \label{eq-u0}
\begin{cases}
(D_t-\Delta_X+q)u^f (x,t)=f(x,t), \quad x\in X,\,t\in \N,\\
u^f(x,t)|_{t=0}=0.
\end{cases}
\end{equation}

We claim the following Duhamel's principle in our case: the solution $u^f$ of the non-homogeneous equation \eqref{eq-u0} is determined by the solutions of the following initial value problems for all $s\in \Z_+$:
\begin{equation}\label{eq-duhamel-s}
\begin{cases}
(D_t-\Delta_X+q)u (x,t)=0, \quad x\in X,\, t\geq s, \, t\in \N, \\
u(x,s)=f(x,s-1), \quad x\in X.
\end{cases}
\end{equation}
More precisely, denote by $u_{(s)}:X \times \{s,s+1,\cdots\} \to \R$ the solution of the equation \eqref{eq-duhamel-s} with the initial value at time $s\in \Z_+$. 
We claim that 
\begin{equation}\label{duhamel-principle}
u^f(x,t)=\sum_{s=1}^t u_{(s)}(x,t),\;\;x\in X,\, t\in \Z_+;\quad u^f(\cdot,0)=0.
\end{equation}
Let us assume the formula \eqref{duhamel-principle} is true for the moment before proving it a bit later.
Since ${\rm supp}(f)\subset B\times \N$, we can write $f$ as $f(x,t)=\sum_{y\in B} f(y,t) \delta_{y}(x)$. Hence the solution $u_{(s)}$ can be written as a linear combination of the data $\{K(y;\cdot,\cdot): y\in B\}$. Therefore the data $\{K(y;z,t):y,z\in B,t\in \N\}$ determines $u^f|_{B \times \Z_+}$ by the formula \eqref{duhamel-principle}.

At last, we prove the claim \eqref{duhamel-principle}.
This is due to the following identities:
\begin{eqnarray*}
D_t u^f(x,t)&=&\sum_{s=1}^{t+1} u_{(s)}(x,t+1)-\sum_{s=1}^{t} u_{(s)}(x,t) \\
&=& u_{(t+1)}(x,t+1)+\sum_{s=1}^{t}\big( u_{(s)}(x,t+1)-u_{(s)}(x,t)\big) \\
&=& f(x,t)+ \sum_{s=1}^{t} D_t u_{(s)}(x,t) \\
&=& f(x,t)+\sum_{s=1}^t (\Delta_X-q) u_{(s)}(x,t) = f(x,t)+ (\Delta_X-q) u^f (x,t).
\end{eqnarray*}
This shows that $u^f$ defined by \eqref{duhamel-principle} satisfies the first equation in \eqref{eq-u0}. Then the claim follows from the uniqueness of the solution of \eqref{eq-u0}.
\end{proof}

\section{Discrete-time random walk} \label{sec-walk}

Let $H_t^{x_0}$, $t\in \N$ be a discrete-time random walk (homogeneous Markov chain) on a finite graph $(X,E)$ starting from the vertex $x_0\in X$ at $t=0$ (i.e. $H_0^{x_0}=x_0$). Denote by $p_{xy}$ the conditional probability \eqref{pxy} of the state changing from $x$ to $y$, that is, $p_{xy}:=\mathbb{P}(H_{t+1}^{x_0}=y\, | \, H_{t}^{x_0}=x )$.
Assume that Condition \ref{condition-Markov} is satisfied.

Given a function $w:X\to \R$, we consider the expectation 
\begin{equation} \label{def-u-section4}
u(x,t):=\mathbb E(w(H_t^{x})).
\end{equation}
By definition,
$$u(x,t)=\sum_{z\in X} \mathbb{P}(H_{t}^x=z) w(z),$$
where $\mathbb{P}(H_{t}^x=z)$ is the probability that the random walk $H_t^x$ is $z$ at time $t$. It immediately follows that $u(x,0)=\mathbb E(w(H_0^{x}))=w(x)$. By the Markov property and Condition \ref{condition-Markov}(1), we have
\begin{eqnarray*}
u(x,t+1) &=& \sum_{z\in X} \mathbb{P}(H_{t+1}^x=z) w(z) \\
&=& \sum_{z\in X} \Big(\mathbb{P}(H_1^x=x)\cdot\mathbb{P}(H_{t}^x=z)+ \sum_{y\sim x, y\in X}\mathbb{P}(H_1^x=y)\cdot\mathbb{P}(H_{t}^y=z) \Big) w(z) \\
&=& p_{xx} u(x,t)+ \sum_{y\sim x, y\in X}p_{xy}u(y,t).
\end{eqnarray*}
Condition \ref{condition-Markov}(1) yields
\begin{equation} \label{sum1}
p_{xx}+\sum_{y\sim x, y\in X} p_{xy}=1,
\end{equation}
which gives
\begin{eqnarray} \label{eq-pxy-rough}
D_t u(x,t) = u(x,t+1)-u(x,t) = \sum_{y\sim x,y\in X} p_{xy} \big(u(y,t)-u(x,t) \big).
\end{eqnarray}

Assuming $p_{xy}=c_{xy}/m(x)$ as in Condition \ref{condition-Markov}(2), the right-hand side of the equation \eqref{eq-pxy-rough} fits into our definition of the graph Laplacian \eqref{Laplaciandef_int} with the choice of weights
\begin{equation} \label{weights-walk}
\mu_x=m(x),\quad g_{xy}=c_{xy}.
\end{equation}
This shows that $u(x,t)$ given by \eqref{def-u-section4} satisfies the discrete-time heat equation, or a Feynman-Kac type formula:
\begin{eqnarray}\label{eq-walk}
\begin{cases} D_t u(x,t)-\Delta_X u(x,t)=0, \quad\hbox{for  }(x,t)\in X \times \N, \\
u(x,t)|_{t=0}=w(x), \quad\hbox{for  }x\in X, \end{cases}
\end{eqnarray}
with the choice of weights \eqref{weights-walk}.

\medskip
Now let us prove Theorems \ref{prop-hitting times walk} and \ref{prop-walk}.

\begin{proof}[Proof of Theorem \ref{prop-hitting times walk}]

Let us assume that $r(t,x,y):=\Prob\{\tau^+(x,y)=t\}$ are known for all $x,y\in B$.
We see that if $H_{T}^{x}=y$, then there are times $1\le t_1\le\dots\le t_j=T$
such that $H_{t}^{x}$ arrives first time at $y$ at time $t_1$, and all
subsequent times when it arrives at $y$ are $ t_2,t_3,\dots,t_j$ with $t_j=T$.
Using this, we see that
\beq\label{moving probabilities}
\Prob(H_{T}^{x}=y)&=&\sum_{j=1}^T \sum_{1\le t_1\le\dots\le t_j=T}
\Prob\bigg(\{t:\ H_{t}^{x}=y,\ t\leq T\}=\{t_1,t_2,\dots,t_j\}\bigg)
\\ \nonumber
&&\hspace{-10mm}=r(T,x,y)+\bigg(\sum_{j=2}^T \sum_{1\le t_1\le\dots\le t_j=T}
r(t_1,x,y)\cdot \big(\prod_{i=2}^{j} r(t_i-t_{i-1},y,y)\big)\bigg).
\eeq
Thus we see that the given data uniquely determine 
\beq\label{new data}
\hbox{the set $B$ and the probabilities $\mathbb{P}(H_{T}^{y}=z)$
 for all $y,z\in B$ and $T\in \N$.}
\eeq
Next we assume that we are given the data \eqref{new data} and we will show that
these data uniquely determine
the graph $(X,E)$ (up to an isometry) and the probabilities $p_{xy}$ for all $x,y\in X$.

We can now modify the random walk so that it is guaranteed to be aperiodic. Let $\tilde H_{t}^x$ be another random walk
defined via formula \eqref{probabilities in c},
that is, the transition probabilities  of $\tilde H_{t}^x$ are given by
\beq\label{probabilities in c2}\tilde p_{xy}=\frac {\tilde c_{xy}}{\tilde m(x)},\quad \tilde m(x)={\sum_{z\sim x\hbox{ \tiny or } z=x} \tilde c_{xz}},
\eeq
where  
\beq\label{tilde c1}
& &\tilde c_{xy}= c_{xy},\quad \hbox{if }x\not =y,\\
& &\tilde c_{xx}= c_{xx}+m(x),\nonumber
\eeq
where $c_{xy}$  and $m(x)$  are parameters in formula \eqref{probabilities in c} for 
$H_{t}^x$. Observe that $ \tilde m(x)=2m(x)$. Roughly speaking, this means that at every time $t$, the random walk $\tilde H_{t}^x$ moves
with probability $\frac 12$ as $H_{t}^x$ and with probability $\frac 12$ stays at the vertex where it was.
To analyze the modified random walk, let $Y_t\in\{0,1\}$ be independent, identically distributed Bernoulli random variables that have the value 1 with probability $\frac 12$. An equivalent way to define $\tilde H_{t}^x$ is to  set that 
if $Y_t=1$ then  $\tilde H_{t}^x$ stays at the vertex  $\tilde H_{t-1}^x$;
if $Y_t=0$ then $\tilde H_{t}^x$ moves from the vertex $z=\tilde H_{t-1}^x$ in the same way as the random walk $H_{t}^x$ would move from
the vertex $z$.

We see that the probability distributions of  the arrival times $\tilde r(t,x,y)$ for the random walk $\tilde H_{t}^x$ are
\beq
& &\tilde  r(t,x,y)=\sum_{n=0}^{t-1} 
\Prob\big(\{\sum_{j=1}^{t-1} Y_j=n\hbox{ and }Y_t=0\} \big)\cdot \Prob \big(\{\tau^+(x,y)=t-n\} \big)
,\quad \hbox{if }x\not =y,\\
& &\tilde  r(t,x,x)=
\Prob(\{Y_1=1\})+
\Prob(\{Y_1=0\})\cdot \Prob(H^x_1=x)
 ,\quad \hbox{if }t=1,
 \eeq
 and for $t\ge 2$,
\ba
& &\tilde  r(t,x,x)=\sum_{n=0}^{t-2}
\Prob(\{Y_1=0\})\cdot \Prob \big(\{
\sum_{j=2}^{t-1} Y_j=n
\hbox{ and }Y_t=0\} \big)\cdot \Prob \big(\{\tau^+(x,x)=t-n\} \big).
\ea
In other words,
\beq
& &\tilde  r(t,x,y)=\sum_{n=0}^{t-1} q(t,n)r(t-n,x,y),\quad \hbox{if }x\not =y,\\
& &\tilde  r(t,x,x)=\frac 12+\frac 12r(1,x,x) ,\quad \hbox{if } t=1,\\
& &\tilde  r(t,x,x)=\sum_{n=0}^{t-2} \frac 12\cdot q(t-1,n)r(t-n,x,x) ,\quad \hbox{if } t\ge 2,
\eeq
where 
\beq
q(t,n)=\Prob \big(\{\sum_{j=1}^{t-1} Y_j=n\hbox{ and }Y_t=0\} \big)= \binom{t-1}{n}2^{-t}.
\eeq
Thus, using the formula \eqref{moving probabilities}, we see that  $r(t,x,y)$, $x,y\in B$
determine $\Prob(\tilde H_{t}^{x}=y)$ and
$\tilde p_{xy}=\Prob(\tilde H_{1}^{x}=y)$ for all $x,y\in B$. 

Observe that the random walk $\tilde H_{t}^{x}$ is irreducible, aperiodic
and recurrent. Thus by \cite[Theorem 7.18]{Kallenberg}, the random walk $\tilde H_{t}^{x}$ has
a unique stationary state $(\tilde \pi_y)_{y\in X}$ and 
\beq
\tilde \pi_y=\lim_{t\to \infty} \Prob(\tilde H_{t}^{x}=y).
\eeq
Moreover, we see that
\beq\label{s-formula}
s_y:=\frac {\tilde m(y)}{\sum_{z\in X} \tilde m(z)}
\eeq
satisfies
\beq
\sum_{y\in X} \tilde p_{yx}s_y=
\sum_{y\in X} \frac{\tilde c_{yx}}{\tilde m(y)}
\frac {\tilde m(y)}{\sum_{z\in X} \tilde m(z)}
=
\frac{\sum_{y\in X} \tilde c_{yx}}{\sum_{z\in X} \tilde m(z)}
=
\frac {\tilde m(x)}{\sum_{z\in X} \tilde m(z)}=
s_x,
\eeq
and hence $(s_y)_{y\in X}$
is a stationary state of the random walk $\tilde H_{t}^{x}$. As the stationary state is unique,
we see that $s_y= \tilde \pi_y$.
This implies that $  r(t,x,y)$, $x,y\in B$
determine $(s_y)_{y\in B}$.

Observe that formula \eqref{s-formula} implies
for any $x,y\in B$,
\beq\label{ratio of ms}
\frac {\tilde m(x)}
 {\tilde m(y)}=\frac {s_x}{s_y}=\frac {\tilde \pi_x}{\tilde \pi_y}
 \eeq
 can be determined when $B$ and  $  r(t,x,y)$, $x,y\in B$ are given.
To normalize the parameters, let
$A_0=\frac 1{\tilde m(x_0)}>0$ with some $x_0\in B$. Then
 the formula \eqref{ratio of ms} determines
 $A_0 {\tilde m(x)}$ for all $x\in B$.
Therefore from \eqref{probabilities in c2}, we can determine the weights 
 $\tilde c_{xy}$ for $x,y\in B$, $x\sim y$, up to a multiple of constant, i.e. the numbers
 \beq\label{weights in c2}
{A_0}\tilde c_{xy}=
 \tilde p_{xy}\cdot 
A_0 {\tilde m(x)}.
\eeq
 Finally, using formulae \eqref{tilde c1},
 we can determine the numbers ${A_0}c_{xy},\, A_0 m(x)$  for all $x,y\in B$.
 
\medskip
Now we show that the knowledge of $r(t,x,y)$, $A_0 m(x)$ for all $x,y\in B,\, t\in \Z_+$ determines the graph structure and the transition matrix of the random walk.
Recall that the expectation $u(x,t)=\mathbb E(w(H_t^{x}))$ satisfies the discrete-time heat equation \eqref{eq-walk}. Since we only know the weights $\mu,g$ in \eqref{weights-walk} up to a gauge $A_0>0$, we consider a new choice of weights in \eqref{eq-walk}:
\begin{equation}\label{weights-walk-gauge}
\tilde{\mu}_x=A_0 m(x),\quad \tilde{g}_{xy}=A_0 c_{xy}.
\end{equation}
Note that the gauge transformation $(\mu,g) \mapsto (\tilde{\mu},\tilde{g})$ does not change the graph Laplacian \eqref{Laplaciandef_int} or the equation \eqref{eq-walk}.

When we take the initial condition $w(x)=\delta_{y}(x)$ in the equation \eqref{eq-walk}, we see that 
the solution 
$u(x_0,t)$ of  \eqref{eq-walk} is 
\ba
u(x_0,t)=\mathbb E(\delta_{y}(H_{t}^{x_0})) =\mathbb{P}(H_{t}^{x_0}=y).
\ea 
Varying $x_0,y\in B,\,t \in \Z_+$ shows that by knowing $r(t,x,y)$ for all $x,y\in B,\, t\in \Z_+$, the formula \eqref{moving probabilities} determines the values $u|_{B\times \N}$ for the equation \eqref{eq-walk} with initial values of the form $w(x)=\delta_{y}(x)$ for any $y\in B$. Then Lemma \ref{lemma-walk-data} reduces the inverse problem for random walk to the inverse problem (iv). If the Two-Points Condition is further assumed, then the unique continuation for eigenfunctions holds due to Proposition \ref{prop-uc}. Hence the inverse problem for random walk is reduced to the inverse interior spectral problem (i) by Theorem \ref{Thm-equiv}, considering that $\tilde{\mu}|_B=A_0 m|_B$ is known. Then Theorem \ref{IISP}(2) enables us to recover the graph $(X,E)$ and weights $\tilde{\mu},\tilde{g}$ in \eqref{weights-walk-gauge}, and consequently recover 
$$A_0 c_{xx}=A_0 m(x)-\sum_{y\sim x} A_0 c_{xy}=\tilde{\mu}_x-\sum_{y\sim x}\tilde{g}_{xy}.$$
At last, the probabilities $p_{xy}$ for all $x,y\in X$ are uniquely determined by formula \eqref{probabilities in c}.
\end{proof}

\smallskip
\begin{proof}[Proof of Theorem \ref{prop-walk}]
Suppose we are able to observe the random walk on a subset $B\subset X$ of vertices. We consider  a single  realization of the random walk $H_t^{x_0},\, t\in \N$, that is, one sample of the random walk process. Assume the initial position $x_0$ is unknown. Let $y,z\in B$ and $T\in \Z_+$. 

Let 
\ba
\tau_1(x_0,y)=\inf \big\{t\ge 1: \ H_t^{x_0}=y \big\}
\ea  
be the first passing time of the random walk $H_t^{x_0}$  through the point $y$. Moreover, we define
recursively the random variables
\ba
\tau_n(x_0,y)=\inf \big\{t>\tau_{n-1}(x_0,y): \ H_t^{x_0}=y \big\},\quad n=2,3,\dots,
\ea 
to
be the $n$:th  passing time of the random walk $H_t^{x_0}$ through the point $y$.
As $X$  is finite and connected, we have $\tau_n(x_0,y)<\infty$ almost surely for any $n$.

\smallskip
 Let 
 $n\ge 1$  and denote $\tau_n=\tau_n(x_0,y)$ so that $H_{\tau_n}^{x_0}=y$. 
Note that each $\tau_n$ is a stopping time with respect to the random walk $H_t^{x_0},\,t\in \N$.
By the strong Markov property of random walks, see
\cite[Theorem 6.2.24]{B} (or \cite[Theorem 1.4.2]{Norris}), it holds for the stopping time
 $\tau_n$ that 
 the random variables $H_{\tau_n+t}^{x_0}$, $t\ge 1$ are independent of the
random variables $H_{t}^{x_0}$, $t=1,2,\dots,\tau_n-1$.
Moreover, the random process $t\to H_{\tau_n+t}^{x_0}$ has the same distribution
as the random walk $t\to H_{t}^{y}$. 
In particular, as $\tau_{_{2T(n+1)}}> \tau_{_{2Tn}}+T$, this implies that the events
\beq
\mathcal{A}_{n}:=\big\{H_{\tau_{_{2Tn}}+T}^{x_0}=z \big\},\quad n=1,2,\dots,
\eeq
are independent, and $\Prob(\mathcal{A}_{n})=\Prob(H_T^y=z)$ for any $n$.
Let ${\bf 1}_{\mathcal A}$ denote the indicator function of an event $\mathcal A$. 
Then the random variables ${\bf 1}_{\mathcal{A}_{n}}$, $n\in \Z_+$ are independent.
Thus the strong law of large numbers, see e.g.
\cite[Theorem 3.23]{Kallenberg}, implies that we have almost surely
\beq\label{walk-prob}
\mathbb{P}(H_{T}^{y}=z)
=\lim_{N\to\infty}
\frac 1{N}\sum_{n=1}^N {\bf 1}_{\mathcal{A}_{n}}.
\eeq
The right-hand side of \eqref{walk-prob} is determined by
the random process $\hat H_t^{x_0}$ given in \eqref{process hat H}.

 Summarizing, we have shown that  the given data determine the probabilities $\mathbb{P}(H_{T}^{y}=z)$
 for all $y,z\in B$ and $T\in \N$. As $B$ is included in the given data, we have 
determined the data \eqref{new data}. As shown in the proof of Theorem \ref{prop-hitting times walk},
the data \eqref{new data} uniquely determine
the graph $(X,E)$ (up to an isometry) and the probabilities $p_{xy}$ for all $x,y\in X$.
\end{proof}

\section{Applications of the unique continuation} \label{section-uc}

Let $(X,E)$ be a finite weighted graph and $B\subset X$ be a subset of vertices. Let $q:X\to \R$ be a potential function. Assume the graph satisfies the unique continuation for eigenfunctions of the Schr\"odinger operator $-\Delta_X+q$, that is, there does not exist a nonzero eigenfunction vanishing identically on $B$. We will apply the unique continuation assumption to study the observability and controllability at $B$ and prove Theorem \ref{Thm-observe-control}.

\subsection{Observability} \hfill
 
\smallskip 
Let $\Phi_w$  be the solution of the following continuous-time heat equation
\begin{eqnarray}
& &\frac \p{\p t} \Phi(x,t)-\Delta_X \Phi(x,t)+q(x) \Phi(x,t)=0, \quad\hbox{for  }(x,t)\in X \times \R_{\geq 0}, \label{eq-cont-heat-initial} \\
& &\Phi(x,t)|_{t=0}=w(x), \quad\hbox{for  }x\in X, \label{eq-cont-heat-initial-condition}
\end{eqnarray}
where the initial data $w(x)$ is unknown. The following lemma shows that if we know the interior spectral data,
it is possible to determine the initial condition of a  state from the measurement of the solution on $B$.

\begin{lemma}\label{observe-cont-heat}
Let $(X,E)$ be a finite weighted graph and $B\subset X$.
Assume that there does not exist a nonzero eigenfunction, of the Schr\"odinger operator $-\Delta_X+q$, vanishing identically on $B$. Suppose we are given the interior spectral data $\big(\lambda_j,\phi_j|_B\big)_{j=1}^{|X|}$. Then the measurement $\Phi_w|_{B\times \R_{+}}$ determines $\langle w,\phi_j \rangle_{L^2(X)}$ for all $j=1,\cdots,|X|$.
\end{lemma}
\begin{proof}
The solution of the equation \eqref{eq-cont-heat-initial} with the initial condition \eqref{eq-cont-heat-initial-condition} can be written as 
$$
\Phi_w(x,t)=\sum_{j=1}^{N} e^{-\lambda_j t} \widehat{w}_j \phi_j(x),
$$
where $N=|X|$, and
$$
 \widehat w_j=\langle w,\phi_j\rangle_{L^2(X)}=\sum_{x\in X}w(x)\phi_j(x) .
$$

For simplicity, assume that all eigenvalues have multiplicity one. Suppose we have ordered the eigenvalues such that $\lambda_j<\lambda_{j+1}$ for all $j$. Suppose we have found $\widehat w_j$  for $j=1,2,\dots,J-1$. Using the values of $\Phi_w$ on $B$, we can compute the following limit at any $z\in B$,
\begin{eqnarray*}
b_J(z):=\lim_{t\to+\infty}  e^{\lambda_{J} t} 
\bigg(\Phi_w(z,t)-\sum_{j=1}^{J-1} e^{-\lambda_j t} \widehat w_j \phi_j(z)\bigg).
\end{eqnarray*}
On the other hand,
$$b_J(z)=\lim_{t\to+\infty} \sum_{j=J}^{N} e^{(\lambda_J-\lambda_j) t} \widehat w_j \phi_j(z) =\widehat w_J \phi_J(z).$$
As each $\phi_j$ does not vanish identically on $B$ by assumption,
this shows we can find $ \widehat w_J$. By induction, we can find 
$ \widehat w_j$  for all $j=1,2,\cdots,N$.

For the case where eigenvalues may have multiplicity larger than one.
The above procedure determines
$$
b_J(z)=\sum_{k\in L_J} \widehat{w}_k \phi_k(z), \quad \forall z\in B,
$$ 
where $L_J=\{k:\lambda_k=\lambda_J\}$.
The unique continuation for eigenfunctions implies that $\{\phi_k|_{B}\}_{k\in L_J}$ are linear independent. Thus the quantities $\big(b_J(z)\big)_{z\in B}$ uniquely determine the coefficients $\widehat w_k$ for all $k\in L_J$.
\end{proof}

Analogously, one can also consider the observability for the discrete-time heat equation.
Denote by $U_w$ the solution of the discrete-time heat equation
\begin{eqnarray}
& &D_t U(x,t)-\Delta_X U(x,t)+q(x) U(x,t)=0, \quad\hbox{for  }(x,t)\in X \times \N, \label{eq-discrete-heat-initial} \\
& &U(x,t)|_{t=0}=w(x), \quad\hbox{for  }x\in X, \label{eq-discrete-heat-initial-condition}
\end{eqnarray}
where the initial data $w(x)$ is unknown.

\begin{lemma}\label{observe-discrete-heat}
Let $(X,E)$ be a finite weighted graph and $B\subset X$.
Assume that there does not exist a nonzero eigenfunction, of the Schr\"odinger operator $-\Delta_X+q$, vanishing identically on $B$. Suppose we are given the interior spectral data $\big(\lambda_j,\phi_j|_B\big)_{j=1}^{|X|}$. Then the measurement $U_w|_{B\times \N}$ determines $\langle w,\phi_j \rangle_{L^2(X)}$ for all $j=1,\cdots,|X|$.
\end{lemma}

\begin{proof}
The solution of the equation \eqref{eq-discrete-heat-initial} with the initial condition \eqref{eq-discrete-heat-initial-condition} can be written as
$$
U_w(x,t)=\sum_{j=1}^{N} (1-\lambda_j)^t \widehat{w}_j \phi_j(x),
$$
where $N=|X|$ and $\widehat w_j=\langle w,\phi_j\rangle_{L^2(X)}$.

Assume that all eigenvalues have multiplicity one. Suppose the eigenvalues are ordered such that $|1-\lambda_j|\geq |1-\lambda_{j+1}|$ for all $j$. Denote $N_1=\max\{j: \lambda_j \neq 1\}$.
Suppose we have found $\widehat w_j$ for $j=1,2,\dots,J-1$ when $J\leq N_1$. 
Using the values of $U_w$ on $B$, we can evaluate the following quantity at any $z\in B$ and $t\in \N$,
$$c_J(z,t):=(1-\lambda_J)^{-t} \bigg(U_w(z,t)-\sum_{j=1}^{J-1} (1-\lambda_j)^t \widehat w_j \phi_j(z)\bigg).$$
On the other hand,
\begin{eqnarray*}
\lim_{t\to +\infty} \big( c_J(z,t)+c_J(z,t+1)\big) &=& \lim_{t\to +\infty} \bigg(\sum_{j=J}^{N} \big(\frac{1-\lambda_j}{1-\lambda_J} \big)^t\widehat w_j \phi_j(z) + \sum_{j=J}^{N} \big(\frac{1-\lambda_j}{1-\lambda_J} \big)^{t+1} \widehat w_j \phi_j(z) \bigg) \\
&=& 2\widehat w_J \phi_J(z) + \lim_{t\to +\infty} \sum_{1-\lambda_j=\lambda_J-1} \big((-1)^{t}+(-1)^{t+1}\big)\widehat w_j \phi_j(z) \\
&& + \lim_{t\to +\infty} \sum_{\substack{|1-\lambda_j|\neq|1-\lambda_J| \\ j\geq J+1}} \Big( \big(\frac{1-\lambda_j}{1-\lambda_J} \big)^t+\big(\frac{1-\lambda_j}{1-\lambda_J} \big)^{t+1} \Big)\widehat w_j \phi_j(z).
\end{eqnarray*}
Due to our ordering of the eigenvalues, the last sum tends to $0$. Hence we get
$$\lim_{t\to +\infty} \big( c_J(z,t)+c_J(z,t+1)\big)=2\widehat w_J \phi_J(z).$$
As each $\phi_j$ does not vanish identically on $B$ by assumption,
this shows we can find $ \widehat w_J$. By induction, we can find 
$ \widehat w_j$  for all $j=1,2,\cdots,N_1$. 

If there are eigenvalues with multiplicity larger than one,
the procedure above determines $\sum_{k\in L_J} \widehat{w}_k \phi_k(z)$, where $L_J=\{k:\lambda_k=\lambda_J\}$. Then the same argument as the last part of Lemma \ref{observe-cont-heat} makes it possible to determine the coefficients $\widehat{w}_k$.

It remains to determine the coefficients corresponding to eigenvalue $1$. Since we have already determined all the coefficients $\widehat{w}_j$ corresponding to eigenvalues not equal to $1$, we have the knowledge of the following quantity
$$U_w(z,0)-\sum_{\lambda_j\neq 1} \widehat{w}_j \phi_j(z)=\sum_{\lambda_j= 1} \widehat{w}_j \phi_j(z),\quad \forall z\in B.$$
In the same way, the unique continuation for eigenfunctions makes it possible to determine all the coefficients $\widehat{w}_j$ corresponding to eigenvalue $1$.
\end{proof}

\smallskip
\subsection{Controllability} \hfill

\medskip
Let $U^f: X\times \{0,1,\cdots,T\}\to \R$ be the solution of the following non-homogeneous discrete-time heat equation up to time $T$,
\begin{eqnarray}
& &D_t U(x,t)-\Delta_X U(x,t)+q(x)U(x,t)=f(x,t), \quad\hbox{for  }(x,t)\in X \times \{0,1,\cdots,T\}, \label{non-homoge} \\
& &U(x,t)|_{t=0}=0, \quad\hbox{for  }x\in X \label{non-homoge-initial}.
\end{eqnarray}
Note that $U^f(\cdot,T)$ is uniquely determined by solving the equations \eqref{non-homoge} for $t=0,1,\cdots,T-1$.
We consider the situation where $f(\cdot,t)$ is real-valued and is only supported on a subset $B\subset X$ for all $t\in \{0,1,\cdots,T-1\}$.

\begin{proposition}\label{prop-control}
Let $(X,E)$ be a finite weighted graph and $B\subset X$.
Assume that there does not exist a nonzero eigenfunction, of the Schr\"odinger operator $-\Delta_X+q$, vanishing identically on $B$.
Then for any $T\geq |X|,\,T\in \Z_+$, we have
$$\big\{U^f(\cdot,T) : \,\supp(f)\subset B\times \{0,1,\cdots,T-1\} \big\}=L^2(X).$$
\end{proposition}

\begin{proof}
Since $L^2(X)$ is finite dimensional,
it suffices to prove the following statement: if a function $v:X\to \R$ satisfies $\langle U^f(\cdot,T),v \rangle_{L^2(X)}=0$ for all $f: X\times \{0,1,\cdots,T-1\} \to \R$ with support on $B\times \{0,1,\cdots,T-1\}$, then $v=0$. 

Consider the following discrete-time equation
\begin{eqnarray}
& &D_t^{\ast} \psi(x,t)-\Delta_X \psi(x,t)+q(x)\psi(x,t)=0, \quad\hbox{for  }(x,t)\in X \times \{1,2,\cdots,T\}, \label{eq-adjoint} \\
& &\psi(x,t)|_{t=T}=v(x),  \quad\hbox{for  }x\in X, \label{eq-endstate}
\end{eqnarray}
where
\begin{equation}\label{Dt-adjoint}
D_t^{\ast} \psi(x,t):=\psi(x,t-1)-\psi(x,t).
\end{equation}
Observe that the equation \eqref{eq-adjoint} is defined by solving backwards in time from $t=T$, and thus the solution $\psi(\cdot,t)$ uniquely extends to $t=0$.
Then by \eqref{non-homoge} and \eqref{eq-adjoint},
\begin{eqnarray*}
\sum_{t=0}^{T-1} \big\langle \psi(\cdot,t+1),f(\cdot,t) \big\rangle_{L^2(X)} &=& \sum_{t=0}^{T-1} \Big(\big\langle \psi(\cdot,t+1),(D_t -\Delta_X+q)U^f(\cdot,t) \big\rangle_{L^2(X)} \\
&&- \big\langle U^f(\cdot,t),(D_t^{\ast} -\Delta_X+q)\psi(\cdot,t+1) \big\rangle_{L^2(X)} \Big) \\
&=& \sum_{t=0}^{T-1} \Big(\big\langle \psi(\cdot,t+1),D_t U^f(\cdot,t) \big\rangle - \big\langle U^f(\cdot,t),D_t^{\ast}\psi(\cdot,t+1) \big\rangle \Big),
\end{eqnarray*}
where we have used the fact that $\Delta_X$ is self-adjoint on $L^2(X)$ (as the graph has no boundary).
Hence by definitions \eqref{eq-discrete-derivative}, \eqref{Dt-adjoint}, initial conditions \eqref{non-homoge-initial}, \eqref{eq-endstate} and the assumption, we have
\begin{eqnarray*}
\sum_{t=0}^{T-1} \big\langle \psi(\cdot,t+1),f(\cdot,t) \big\rangle_{L^2(X)} &=& \sum_{t=0}^{T-1} \Big(\big\langle \psi(\cdot,t+1),U^f(\cdot,t+1)-U^f(\cdot,t) \big\rangle \\
&&- \big\langle U^f(\cdot,t),\psi(\cdot,t)-\psi(\cdot,t+1) \big\rangle \Big) \\
&=& \sum_{t=0}^{T-1} \Big(\big\langle \psi(\cdot,t+1),U^f(\cdot,t+1) \big\rangle - \big\langle U^f(\cdot,t),\psi(\cdot,t)\big\rangle \Big) \\
&=& \langle U^f(\cdot,T),v \rangle=0.
\end{eqnarray*}
Since $f$ is arbitrary and is supported on $B\times \{0,1,\cdots,T-1\}$, the formula above implies that $\psi|_{B\times \{1,2,\cdots,T\}}=0$. Next, we show that $\psi=0$ on $X\times \{0,1,2,\cdots,T\}$ and hence $v=0$.

\smallskip
Let $\{\phi_j\}_{j=1}^N$ be a choice of orthonormalized eigenfunctions of $-\Delta_X+q$ corresponding to the eigenvalues $\lambda_j$, where $N=|X|$.
By the equation \eqref{eq-adjoint}, the solution $\psi$ can be written as
\begin{equation}\label{heatexpansion}
\psi(x,t)=\sum_{j=1}^N (1-\lambda_j)^{T-t} a_j \phi_j(x), \;\textrm{ for } t=0,1,\cdots,T,
\end{equation}
where $a_j=\langle \psi(\cdot,T),\phi_j \rangle_{L^2(X)}=\langle v,\phi_j \rangle_{L^2(X)}$.

Now assume that there does not exist a nonzero eigenfunction vanishing identically on $B$. For simplicity, assume all eigenvalues have multiplicity one. The conditions $\psi|_{B \times \{1,2,\cdots,T\}}=0$ and $T\geq N$ give at least $N$ number of linear equations with $N$ variables $\{a_j\}_{j=1}^{N}$ at any $z\in B$. 
\begin{equation}\label{eq-linear-system}
\sum_{j=1}^{N} (1-\lambda_j)^{T-t} a_j\phi_j(z)=0,\; \textrm{ for }t=1,2,\cdots,T.
\end{equation}
We pick the latter $N$ linear equations (i.e. for $t=T-N+1,\cdots,T$) and rewrite these equations at any $z\in B$ in the following matrix form:
\[
  \left( {\begin{array}{cccc}
  1&1&\cdots &1 \\
  1-\lambda_1 & 1-\lambda_2 & \cdots & 1-\lambda_N \\
  (1-\lambda_1)^2 & (1-\lambda_2)^2 & \cdots & (1-\lambda_N)^2 \\
  \cdots &\cdots & \cdots &\cdots \\
  (1-\lambda_1)^{N-1} & (1-\lambda_2)^{N-1} & \cdots & (1-\lambda_N)^{N-1}
  \end{array} } \right) 
  \left( \begin{array}{c} 
  a_1\phi_1(z) \\
  a_2\phi_2(z) \\
  \cdots\\
  a_N\phi_N(z) 
  \end{array}\right)=0.
\]
The square matrix above (a Vandermonde matrix) is invertible, assuming all $\lambda_j$ are distinct:
\[
   \left| {\begin{array}{cccc}
  1&1&\cdots &1 \\
  1-\lambda_1 & 1-\lambda_2 & \cdots & 1-\lambda_N \\
  (1-\lambda_1)^2 & (1-\lambda_2)^2 & \cdots & (1-\lambda_N)^2 \\
  \cdots &\cdots & \cdots &\cdots \\
  (1-\lambda_1)^{N-1} & (1-\lambda_2)^{N-1} & \cdots & (1-\lambda_N)^{N-1}
  \end{array} } \right| = \prod_{1\leq j< l \leq N} (\lambda_j-\lambda_l) \neq 0.
\]
Hence we get $a_j\phi_j(z)=0$ for all $z\in B$ and $j=1,\cdots,N$. As each $\phi_j$ is not identically zero on $B$, then $a_j=0$ for all $j=1,\cdots,N$.

If $\lambda_j$ are not all distinct, then the process above reduces to the non-degeneracy of the Vandermonde matrix of a reduced dimension, and yields that for each $j$,
$$\sum_{k\in L_j} a_k \phi_k(z)=0,\quad \forall z\in B,$$ 
where $L_j=\{k:\lambda_k=\lambda_j\}$. Since there does not exist a nonzero eigenfunction vanishing identically on $B$ by assumption, the vectors $\{\phi_k|_{B}\}_{k\in L_j}$ for each $j$ are linear independent. Hence $a_k=0$ for all $k\in L_j$.
\end{proof}

\smallskip
\appendix

\section{Isospectral graphs with identical interior spectral data} \label{counterexample_eigenvector}

In this appendix, we calculate the eigenvalues and corresponding eigenfunctions (eigenvectors) for the isospectral graphs in Figure \ref{fig_isospectral}. It is known in \cite{FK,Tan} that these graphs have the same eigenvalues, that is, the graphs
are isospectral. 
In this appendix, we show that their interior spectral data are actually identical on the subset $B=\{v_1,v_2\}$ of blue vertices. The operator in question is the combinatorial Laplacian, i.e. setting $\mu,g \equiv 1$ in our definition of the graph Laplacian \eqref{Laplaciandef_int}. As a consequence, this phenomenon provides a counterexample to the solvability of an inverse problem for random walks, as explained in the following lemma.

\begin{lemma}\label{lemma-isospectral}
Let $(X,E)$ and $(\bar{X},\bar{E})$ be the two graphs in Figure \ref{fig_isospectral}. We consider the graph Laplacian \eqref{Laplaciandef_int} with weights $\mu\equiv C,\, g\equiv 1$, where $C>0$ is a constant.
Then there exist subsets $B\subset X$ and $\bar{B} \subset \bar{X}$ with $|B|=|\bar{B}|=2$ (marked blue in the figure) such that the following statements hold.
\begin{itemize}
\item[(1)] Let $\Phi:B\to \bar{B}$ be a bijection. Then we can find a complete orthonormal family of eigenfunctions $\phi_j,\, \bar{\phi}_j$ for each graph corresponding to eigenvalue $\lambda_j,\, \bar{\lambda}_j$, such that $\lambda_j=\bar{\lambda}_j$ and $\phi_j|_B = \bar{\phi}_j \circ \Phi|_B$ for all $j$.
\item[(2)] Let $C\geq 4$ and $\Phi:B\to \bar{B}$ be a bijection. Then the random walk process given by \eqref{weights-to-walk} satisfies 
$$\mathbb{P} \big(H_{t}^{y}=z \big)=\mathbb{P}\Big(\bar{H}_{t}^{\Phi(y)}=\Phi(z) \Big)$$
 for all $y,z\in B$ and $t\in \N$. Here $H_t^y,\, t\in \N$ denotes the random walk on $(X,E)$ starting from $y$, and $\bar{H}_t^{\Phi(y)}$ denotes the random walk on $(\bar{X},\bar{E})$ starting from $\Phi(y)$.
\end{itemize}
\end{lemma}

\begin{proof}
(1) Observe that it suffices to prove the first claim for the combinatorial Laplacian, i.e. $\mu,g \equiv 1$. This is because changing $\mu$ from $1$ to $C$ only changes eigenvalues by a factor of $C^{-1}$, and changes normalized eigenfunctions by a factor of $C^{-1/2}$ due to normalization \eqref{innerproduct_int}. Thus we assume $C=1$ for the first claim without loss of generality.

\smallskip
We label the vertices of the two graphs in Figure \ref{fig_isospectral} as follows. First for the graph on the left, we name the upper left vertex $v_3$, the lower middle vertex $v_4$, on the right from top to bottom $v_5,\, v_6$. Then for the graph on the right, we name the lower vertices from left to right $v_3,v_5,v_6$, and the upper right vertex $v_4$. 

For the combinatorial Laplacian, the $L^2(X)$-norm \eqref{innerproduct_int} of a function on $X$ is simply the usual length of the function as a vector in $\mathbb{R}^{|X|}$. 
In the standard basis of functions $\{e_i\}_{i=1}^6$ on $X$ with $e_i(v_k)=\delta_{ik}$, the combinatorial Laplacians $\Delta$ are the following matrices:

$$\Delta(\textrm{left})=\begin{bmatrix} 
-2&0&1&1&0&0 \\
0&-2&1&1&0&0 \\
1&1&-2&0&0&0 \\
1&1&0&-4&1&1 \\
0&0&0&1&-2&1 \\
0&0&0&1&1&-2
\end{bmatrix}, $$

$$\Delta(\textrm{right})=\begin{bmatrix} 
-2&0&1&1&0&0 \\
0&-2&1&1&0&0 \\
1&1&-3&0&1&0 \\
1&1&0&-3&1&0 \\
0&0&1&1&-3&1 \\
0&0&0&0&1&-1
\end{bmatrix}. $$

\noindent After direct calculations, we list their eigenvalues and (a choice of) orthonormalized eigenvectors as follows. We write the eigenfunctions $\phi_j$ in the form of vectors $\big(\phi_j(v_1),\cdots,\phi_j(v_6)\big)$. Compare the 1st, 2nd entries of $\phi_j(\textrm{left})$ with the 1st, 2nd entries of $\phi_j(\textrm{right})$: they are identical for all $j$.

\smallskip
\noindent $\bullet \;\; \lambda_1=0 \,$: 
$$\phi_1(\textrm{left})=\phi_1(\textrm{right})=\frac{1}{\sqrt{6}}(1,1,1,1,1,1).$$

\noindent $\bullet \;\; -\lambda_2=-2 \,$: 
$$\phi_2(\textrm{left})=\frac{1}{\sqrt{2}}(-1,1,0,0,0,0),\quad
\phi_2(\textrm{right})=\frac{1}{\sqrt{2}}(-1,1,0,0,0,0).$$

\noindent $\bullet \;\;  -\lambda_3=-\lambda_4=-3 \,$: 
$$\phi_3(\textrm{left})=\frac{1}{\sqrt{8}}(-1,-1,2,-1,0,1),\quad
\phi_3(\textrm{right})=\frac{1}{\sqrt{8}}(-1,-1,1,0,2,-1).$$
$$\phi_4(\textrm{left})=\frac{1}{\sqrt{2}}(0,0,0,0,1,-1),\quad
\phi_4(\textrm{right})=\frac{1}{\sqrt{2}}(0,0,-1,1,0,0).$$
Apply the Gram-Schmidt orthogonalization:
$$\phi_3^{\perp} (\textrm{left})=(-1,-1,2,-1,\frac12,\frac12), \quad \phi_3^{\perp} (\textrm{right})=(-1,-1,\frac12,\frac12,2,-1).$$
Observe that the length of $\phi_3^{\perp} (\textrm{left})$ is equal to that of $\phi_3^{\perp} (\textrm{right})$.

\smallskip
\noindent $\bullet \;\;  -\lambda_5=\sqrt{5}-3 \,$: 
\begin{eqnarray*}
\phi_5(\textrm{left})&=&\frac{\sqrt{2}}{\sqrt{5}(\sqrt{5}-1)}(\frac{-\sqrt{5}+1}{2},\frac{-\sqrt{5}+1}{2},-1,\sqrt{5}-2,1,1) \\
&=& (-\frac{1}{\sqrt{10}},-\frac{1}{\sqrt{10}},\ast,\ast,\ast,\ast), \\
\phi_5(\textrm{right})&=&\frac{\sqrt{2}}{\sqrt{5}(3-\sqrt{5})}(\frac{\sqrt{5}-3}{2},\frac{\sqrt{5}-3}{2},-\sqrt{5}+2,-\sqrt{5}+2,\sqrt{5}-2,1) \\
&=& (-\frac{1}{\sqrt{10}},-\frac{1}{\sqrt{10}},\ast,\ast,\ast,\ast).
\end{eqnarray*}

\noindent $\bullet \;\;  -\lambda_6=-\sqrt{5}-3 \,$: 
\begin{eqnarray*}
\phi_6(\textrm{left})&=&\frac{\sqrt{2}}{\sqrt{5}(\sqrt{5}+1)}(\frac{\sqrt{5}+1}{2},\frac{\sqrt{5}+1}{2},-1,-\sqrt{5}-2,1,1) \\
&=& (\frac{1}{\sqrt{10}},\frac{1}{\sqrt{10}},\ast,\ast,\ast,\ast), \\
\phi_6(\textrm{right})&=&\frac{\sqrt{2}}{\sqrt{5}(3+\sqrt{5})}(\frac{\sqrt{5}+3}{2},\frac{\sqrt{5}+3}{2},-\sqrt{5}-2,-\sqrt{5}-2,\sqrt{5}+2,-1) \\
&=& (\frac{1}{\sqrt{10}},\frac{1}{\sqrt{10}},\ast,\ast,\ast,\ast).
\end{eqnarray*}

\medskip
(2) For the second claim, it is convenient to use the equation \eqref{eq-walk-intro}.
The formula \eqref{eq-pxy-intro} shows that $u(x,t):=\mathbb E(w(H_t^{x}))$ satisfies the heat equation \eqref{eq-walk-intro} with the initial value $w$, where the associated graph Laplacian has weights $\mu\equiv C\geq 4,\, g\equiv 1$.
Let $z\in B$ and take the initial value $w(x)=\delta_z(x)$ in the equation \eqref{eq-walk-intro}-\eqref{eq-walk-initial-intro}. Then we see that 
the solution $u(x,t)$ is 
\ba
u(x,t)=\mathbb E(\delta_{z}(H_{t}^{x})) =\mathbb{P}(H_{t}^{x}=z).
\ea 
On the other hand, we can write the solution $u(x,t)$ in terms of eigenvalues and (orthonormalized) eigenfunctions by using the equation \eqref{eq-walk-intro}. Namely, 
$$
u(x,t)=\sum_{j=1}^{6} (1-\lambda_j)^t \widehat{w}_j \phi_j(x),
$$
where
$$\widehat w_j=\langle w,\phi_j\rangle_{L^2(X)}= \langle \delta_z,\phi_j\rangle_{L^2(X)} = C \phi_j(z). $$
Since the eigenvalues and eigenfunctions of the graph Laplacian with weights $\mu\equiv C,\, g\equiv 1$ coincide for the two graphs on $B$ due to (1), then $\widehat w_j$ and hence $u|_{B\times \N}$ coincide. Thus varying $z\in B$ shows that $\mathbb{P}(H_{t}^{y}=z)$ coincide for the two graphs for all $y,z\in B,\, t\in \N$.
\end{proof}

\end{document}